\documentclass{amsart}
\usepackage[english]{babel}
\usepackage{blindtext}
\usepackage[utf8]{inputenc}
\usepackage{fancyhdr}
\usepackage{hyperref}

\usepackage{amsmath,graphicx}
\usepackage{adjustbox}
\usepackage{amssymb}
\usepackage{tikz, tikz-cd}
\usepackage{fullpage}
\usepackage[all]{xy}
\usepackage[margin=1in]{geometry}
\usepackage{amsthm}
\usepackage{amsfonts}
\usepackage{bbm}
\usepackage{comment}
\usepackage{amsmath}
\usepackage{amsthm}
\usepackage{bbm}
\usepackage{array,multirow} 
\usepackage{color}
\usepackage[utf8]{inputenc}
\usepackage[english]{babel}
\usepackage{float}

\usepackage{mathtools}
\DeclarePairedDelimiter{\ceil}{\lceil}{\rceil}
\DeclarePairedDelimiter{\floor}{\lfloor}{\rfloor}

\usepackage{graphicx}
\usepackage{subfig}
\usepackage{pinlabel}
\usepackage{thmtools}
\usepackage{thm-restate}

\usetikzlibrary{arrows.meta,bending,decorations.markings,intersections} 
\tikzset{
    arc arrow/.style args={%
    to pos #1 with length #2}{
    decoration={
        markings,
         mark=at position 0 with {\pgfextra{%
         \pgfmathsetmacro{\tmpArrowTime}{#2/(\pgfdecoratedpathlength)}
         \xdef\tmpArrowTime{\tmpArrowTime}}},
        mark=at position {#1-\tmpArrowTime} with {\coordinate(@1);},
        mark=at position {#1-2*\tmpArrowTime/3} with {\coordinate(@2);},
        mark=at position {#1-\tmpArrowTime/3} with {\coordinate(@3);},
        mark=at position {#1} with {\coordinate(@4);
        \draw
        (@1) .. controls (@2) and (@3) .. (@4);},
        },
     postaction=decorate,
     }
}
\tikzset{
    new arc arrow/.style args={%
    to pos #1 with length #2}{
    decoration={
        markings,
         mark=at position 0 with {\pgfextra{%
         \pgfmathsetmacro{\tmpArrowTime}{#2/(\pgfdecoratedpathlength)}
         \xdef\tmpArrowTime{\tmpArrowTime}}},
        mark=at position {#1-\tmpArrowTime} with {\coordinate(@1);},
        mark=at position {#1-2*\tmpArrowTime/3} with {\coordinate(@2);},
        mark=at position {#1-\tmpArrowTime/3} with {\coordinate(@3);},
        mark=at position {#1} with {\coordinate(@4);
        \draw
        [-{Stealth[length=#2,bend]}]      
        (@1) .. controls (@2) and (@3) .. (@4);},
        },
     postaction=decorate,
     }
}

\usepackage[textsize=scriptsize]{todonotes}

\newcolumntype{L}{>{$}l<{$}}

\usepackage[ruled,vlined]{algorithm2e}

\SetCommentSty{mycommfont}

\usepackage{enumitem}
\usepackage{bm}

\calclayout

\newif\ifrevision
\newcommand{\change}[1]{%
\ifrevision
\textcolor{blue}{#1}%
\else
#1%
\fi
}

\newtheorem{theorem}{Theorem}[section]
\newtheorem{lemma}[theorem]{Lemma}

\newtheorem{corollary}[theorem]{Corollary}
\newtheorem{conjecture}[theorem]{Conjecture}
\newtheorem{proposition}[theorem]{Proposition}
\theoremstyle{definition}  
\newtheorem{definition} [theorem] {Definition}

\newtheorem{remark} [theorem] {Remark}
\newtheorem{question} [theorem] {Question}

\theoremstyle{definition}

\def\conn {\mathbin{\#}}
\def\cC{\mathcal{C}}
\def\CZhat{\widehat{\cC}_\Z}
\def\CZ{{\cC}_\Z}

\newcommand{\F}{{\mathbb{F}}}

\newcommand{\R}{{\mathbb{R}}}
\newcommand{\Z}{{\mathbb{Z}}}

\newcommand{\surj}{\twoheadrightarrow}

\newcommand{\length}{\text{length}}

\newcommand{\spin}{\text{spin}}
\newcommand{\Sym}{\text{Sym}}

\DeclareMathOperator{\HFK}{HFK}

\DeclareMathOperator{\HF}{HF}

\DeclareMathOperator{\CFK}{CFK}

\DeclareMathOperator{\rank}{rank}

\usepackage[utf8]{inputenc}

\def\HFh {\widehat{\operatorname{HF}}}
\def\CFK {\operatorname{CFK}}
\def\HFK {\operatorname{HFK}}

\def\CFKi {\operatorname{CFK}^\infty}

\def\HFKa {\widehat{\operatorname{HFK}}}

\DeclareMathOperator{\Spin}{Spin}
\DeclareMathOperator{\gr}{gr}

\def\spincs {\mathfrak{s}}
\newcommand{\hH}{{\mathcal{H}}}
\def\wa{{\widetilde \alpha}}
\def\wb{{\widetilde \beta}}

\def \pos{\operatorname{\textit{position}}}
\def \dir{\operatorname{\textit{direction}}}
\def \sid{\operatorname{\textit{side}}}
\def \hei{\operatorname{\textit{height}}}
\def \iseq{\operatorname{\textit{initial\_sequence}}}
\def \fseq{\operatorname{\textit{full\_sequence}}}
\def \maxhei{\operatorname{\textit{max\_height}}}
\def \minhei{\operatorname{\textit{min\_height}}}
\def \dh{\Delta_{\operatorname{\textit{height}}}}
\def \po{\operatorname{\textit{positive}}}
\def \ne{\operatorname{\textit{negative}}}
\def \inters{\operatorname{\textit{intersect}}}
\def \ite{\operatorname{\textit{iteration}}}

\title{(1,1) almost L-space knots}
\author{Fraser Binns}
\address{Department of Mathematics, Princeton University, Princeton, NJ, USA}
\email{fb1673@princeton.edu}
\author{Hugo Zhou}
\address{Mathematics, University of Michigan, Ann Arbor, MI, USA}
\email{hugozhou@umich.edu}

\begin{document}
\revisionfalse     

\begin{abstract}
   We give a diagrammatic characterization of the $(1,1)$ knots in the three-sphere and lens spaces which admit large Dehn surgeries to manifolds with Heegaard Floer homology of next-to-minimal rank. This is inspired by a corresponding result for $(1,1)$ knots which admit large Dehn surgeries to manifolds with Heegaard Floer homology of minimal rank due to Greene-Lewallen-Vafaee.
\end{abstract}
\maketitle
\section{Introduction}

Heegaard Floer homology is a powerful package of invariants due originally to Ozsv\'ath-Szab\'o~\cite{OSclosed}. A central question in the study of Heegaard Floer homology is whether or not any objects defined in terms of Heegaard Floer homology can be given Floer-free interpretations. We seek to address a special case of this question in this paper.

The simplest version of Heegaard Floer homology is a three manifold invariant, denoted $\widehat{\HF}(-)$. If $Y$ be a rational homology sphere then $\widehat{\HF}(-)$ satisfies the following inequality:

\begin{align}\label{rankinequality}
    \rank(\widehat{\HF}(Y))\geq |H_1(Y;\Z)|.
\end{align}

For example see \cite[Lemma 1.6]{lecture}. Here $|H_1(Y;\Z)|$ is the cardinality of $H_1(Y;\Z)${.}  An \emph{L-space} is a rational homology sphere for which inequality~\eqref{rankinequality} is tight. Examples include lens spaces. Conjecturally, L-spaces can be characterized as rational homology spheres either which do not admit taut foliations, or whose fundamental groups are not left orderable \cite[Conjecture 5]{juhaz},~\cite[Conjecture 1]{BoyerGordonWatsonLspacelofg}. {\emph{L-space knots} are the knots that admit Dehn surgeries to L-spaces}. L-space knots have played a key role in work on the Berge conjecture \cite{OSlspace} but are also of wider interest. An \emph{almost L-space} is a rational homology sphere $Y$ for which inequality~\eqref{rankinequality} is ``almost tight" -- that is $\rank(\widehat{\HF}(Y))=|H_1(Y;\Z)|+2$. 
Examples include the Brieskorn sphere $\Sigma(2, 3, 11)$. The term ``almost L-space" was coined by Baldwin-Sivek in their work on the characterizing slopes of the knot $5_2$~\cite{baldwin2022characterizing}. Almost L-spaces have have also been studied in other contexts; for example Lin gives results concerning the Stein fillings of a subclass of almost $L$-spaces~\cite{lin2020indefinite}.\

Almost L-space knots are to almost L-spaces as L-space knots are to L-spaces:

\begin{definition}
    An \emph{almost L-space knot} is a knot $K$ { that admits a surgery to an almost $L$-space and is not an $L$-space knot.}
\end{definition}
 {Note that unlike in~\cite{binns2023cfk} we do not require that the surgery is positive.} Examples of almost L-space knots include the mirror of $5_2$, $T(2,3)\#T(2,3)$, and $(2,4g(K)-3)$-cables of L-space knots. The first author proved that the $\CFKi$ type of an almost L-space knot is either a staircase with a length one box or a so called ``almost staircase'' (See the discussion below Definition \ref{def:almoststaircase})~\cite{binns2023cfk}. The $(2,4g(K)-3)$ cables of L-space knots are the only infinite family of almost L-space knots that have appeared in the literature, to the best knowledge of the authors. We give a {new} infinite family of almost L-space knots in Corollary \ref{cor: kj}.   

In this paper, we study rational homology spheres admitting genus one Heegaard splittings, namely $S^3$ and lens spaces. We consider L-space and almost L-space knots in these manifolds. Note that this definition is broader than the usual definition of (almost) L-space knots, which are required to be knots in $S^3$. See the discussion below Definition \ref{def: staircase}.

A knot $K\subset Y$ is called \emph{a $(1,1)$ knot} if there is a genus one Heegaard splitting of $Y$ such that $K$ intersects each handlebody in a trivial arc. In other words, $(1,1)$ knots are  precisely those that admit a doubly-pointed genus one Heegaard diagram. Such a diagram is called a \emph{$(1,1)$ diagram}. 
The fact that $(1,1)$ knots admit simple Heegaard diagrams makes their knot Floer homology comparatively easy to compute and as a result there are a number of results concerning the knot Floer homology of $(1,1)$ knots in the literature \cite{Ras_homology, goda2005knot, geometry11}. In particular, the knot Floer homology of $(1,1)$ knots is computable~\cite{doyle2005calculating}. Despite this, there is no general formula for the knot Floer homology of $(1,1)$ knots.

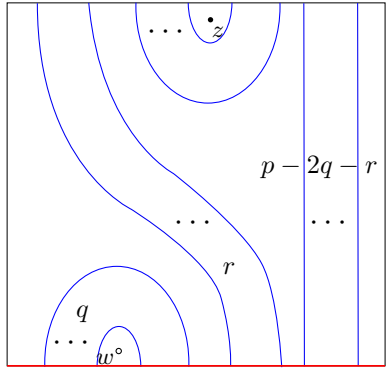
\begin{figure}
\labellist
 \pinlabel {$q$}  at 50 35
 \pinlabel {\Large $\cdots$}  at 46 22
  \pinlabel {$r$}  at 115 55
   \pinlabel {$p - 2q - r$}  at 155 100
   \pinlabel {\Large $\cdots$}  at 88 160
 \pinlabel {\Large $\cdots$}  at 100 75
  \pinlabel {\Large $\cdots$}  at 160 75
 \pinlabel {\small $w$}  at 60 15
 \pinlabel {\small $z$}  at 110 160

\endlabellist
\includegraphics[scale=0.85]{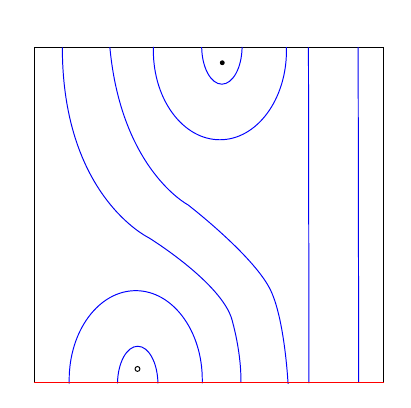}
\caption{The standard $(1,1)$ diagram for the four-tuple $(p,q,r,s)$ by \cite{Ras_homology}, where $q,r\geq 0,$ $2q + r \leq p$ and $0\leq s <p$. There are $q$ strands of ``rainbow arcs'' on both sides. If we label the intersection points on the top and bottom sides from left to right, then the $i$-th point on the top is identified with the $(i+s)$-th point on the bottom.  }
\label{fig: standard}
\end{figure}

Figure \ref{fig: standard} shows a \emph{standard (1,1) diagram} by \cite{Ras_homology}, which can be encoded by a $4$-tuple of integers $(p,q,r,s)$. See the beginning of Section \ref{subsec: 11} for details. Every $(1,1)$ diagram can be isotoped to a unique standard diagram, so we will generally make statements regarding the associated standard diagram. We will be particularly interested in the orientations of the ``rainbow arcs" which are those that bound bigons in Figure~\ref{fig: standard}. If the rainbow arcs around a fixed basepoint are not oriented in the same direction, call the orientation in the minority the \emph{inconsistent direction} and these rainbow arcs \emph{inconsistent arcs.}
(If there are the same number of arcs in each direction, choose {the left} direction as the inconsistent direction.) 
Having no inconsistent arc means that all the rainbow arcs around a fixed basepoint are in the same direction.
\begin{restatable}{definition}{coherent}(\cite[Definition 1.1]{GLV}.)
    A standard diagram is \emph{coherent} if  there are no inconsistent arcs.
\end{restatable}
This diagrammatic condition characterizes L-space knots amongst $(1,1)$ knots:
\begin{restatable}{theorem}{lspace}\label{thm: lspace}(\cite[Theorem 1.2]{GLV}.)
    A standard diagram represents an L-space knot if and only if it is coherent.
\end{restatable}
We seek to extend this result, via an extension of the techniques Greene-Vafaee-Lewallen use to obtain it, to the case of almost L-space $(1,1)$ knots. It would have been natural to conjecture that the new diagrammatic condition required for this would be that there are exactly two inconsistent arcs in the associated standard diagram (one around each basepoint). However, this statement fails even for relatively simple $(1,1)$ diagrams, see Figure \ref{fig: K9204}. It turns out the appropriate definition is the following.

\begin{restatable}{definition}{strongly} \label{def: strongly_almost_coherent}
    A standard diagram is \emph{strongly almost coherent} if {there are at least four rainbow arcs, only two of which --- say $a_1$ and $a_2$ --- have a given orientation and either\begin{enumerate}[label=(\alph*)]
        \item $a_1$ and $a_2$ share a common endpoint;
        \item or $a_1$ and $a_2$ are connected by two other rainbow arcs.
    \end{enumerate}}
\end{restatable}
This condition turns out to be equivalent to having exactly two inconsistent arcs in the universal cover (a condition we call \emph{virtually almost coherent}, see Definition \ref{def: vir_almost}.) We show that this characterizes almost L-space knots amongst $(1,1)$ knots:
\begin{theorem} \label{thm: main}
    A standard diagram represents an almost L-space knot if and only if it is strongly almost coherent.
\end{theorem}
In fact we prove a stronger statement; that the almost L-space knots that arise as $(1,1)$ knots have the $\CFK^\infty$-type of box plus staircase complexes, see Theorem~\ref{thm:(1,1)virtualboxcoherent}. This implies that not every bigraded vector space that arises as $\widehat{\HFK}(K)$ for some $K$ arises as $\widehat{\HFK}(K')$ for some $(1,1)$ knot $K'$; for example the $(2,1)$-cable of $T(2,3)$ which is an almost $L$-space knot but does not have the $\CFK^\infty$-type of a box plus staircase complex. This contrasts with the case of the decategorification of knot Floer homology -- the Alexander polynomial, $\Delta_K(t)$. The Alexander polynomial has the property that any polynomial which arises as $\Delta_K(t)$ for some knot $K$ arises as $\Delta_{K'}(t)$ for some $(1,1)$ knot $K'$~\cite{fujii1996geometric}.


One can readily check by computer algorithm whether or not a given $(1,1)$ diagram is strongly almost coherent; see Section \ref{sec: addendum} for the pseudo code. In Table \ref{ta:list}, we have recorded all the four-tuples with $p\leq 15$ in $S^3$ that satisfies the strongly almost coherent condition. For each pair of the mirror four-tuples $(p,q,r,s)$ and $(p,q,p-2q-r,2q-s)$ exactly one is recorded. 
\begin{table}
   \centering
   \begin{tabular}{|c|c|}\hline
  Four-tuple $(p,q,{r},{s})$  & Knot name \\\hline
  $(5,2,0,1),(5,2,0,4)$ & $4_1$ \\\hline
   $(7,2,0,3),(7,2,0,4),(7,3,0,1)$ & \multirow{2}{*}{$5_2$} \\
   $(7,3,0,2),(7,3,0,5),(7,3,0,6)$ & \\\hline
   $(11,3,1,4)$ & $10_{139}$\\\hline
   $(13,4,1,7)$ & $12n_{725}$\\\hline
   $(15,3,1,4),(15,4,2,5)$ & $16n_{792631}$\\\hline
\end{tabular}\caption{All the $(1,1)$ almost L-space knots  in $S^3$ with $\rank(\HFKa(K))\leq 15$, up to mirroring. }\label{ta:list}
 \end{table}

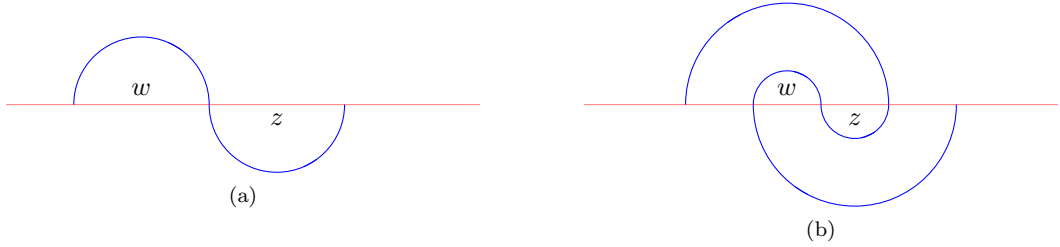
\begin{figure}
\begin{minipage}{.5\linewidth}
     \centering
      \subfloat[]{
       \begin{tikzpicture}[scale=0.9]
       \begin{scope}[thin, red!50!white]
	  \draw  (-1, 0) -- (6,0);
      \end{scope}    
       \draw[arc arrow=to pos 0.7 with length 2mm, blue] (0,0) arc (180:0:1);     
           \draw[arc arrow=to pos 0.7 with length 2mm, blue] (2,0) arc (-180:0:1);   
           
              \node  [above] at (1,0) {$w$};
          \node  [below] at (3,0) {$z$};
    \end{tikzpicture}
     }
\end{minipage}%
\begin{minipage}{.5\linewidth}
     \centering
      \subfloat[]{
       \begin{tikzpicture}[scale=0.9]
       \begin{scope}[thin, red!50!white]
	  \draw  (-1.5, 0) -- (5.5,0);
      \end{scope}
       \draw[arc arrow=to pos 0.7 with length 2mm, blue] (1,0) arc (180:0:0.5);
         \draw[arc arrow=to pos 0.7 with length 2mm, blue] (2,0) arc (-180:0:0.5);
           \draw[arc arrow=to pos 0.7 with length 2mm, blue] (3,0) arc (0:180:1.5);
           \draw[arc arrow=to pos 0.7 with length 2mm, blue] (4,0) arc (0:-180:1.5);
 
              \node  [above] at (1.5,0) {$w$};
          \node  [below] at (2.5,0) {$z$};
         
    \end{tikzpicture}
     }
\end{minipage}%
\label{}
\caption{Two possible cases in Definition \ref{def: strongly_almost_coherent}. The inconsistent arcs are the outermost arcs in the diagram on the right.}
\end{figure}\label{fig: for_def_strongly}
We also give a new infinite family of almost L-space knots. Let $K_j$ be the $(1,1)$ knot in $S^3$ given by the four-tuple $(7+4j,3,4j,2)$ with $j\in \Z_{\geq 0}.$ {Note that these knots are distinct since $\rank(\widehat{\HFK}(K_j))=7+4j$.}

\begin{corollary}\label{cor: kj}
The knot $K_j$ is an almost L-space knot for $j\in \Z_{\geq 0}.$
\end{corollary}
 {To the best of the authors knowledge, this is the first infinite family to occur in the literature with the $\CFKi$-type of the direct sum of a staircase and a box as opposed to an almost staircase, see Section~\ref{sec:Bakground} for relevant definitions.}
\begin{proof}
   The diagram given by the four-tuple $(7+4j,3,4j,2)$ (See Figure \ref{fig: kj} \subref{subfig: kj_a}) is strongly almost coherent according to Definition \ref{def: strongly_almost_coherent}. Therefore $K_j$ is an almost L-space knot by Theorem \ref{thm: main}.
\end{proof}
 Let $\CZhat$ denote the group of pairs $(Y,K)$, such that $K$ is a knot in an integer homology sphere $Y$ which bounds integer homology ball. The group action is given by connected sum. Let $\CZ$ denote the subgroup of $\CZhat$ consisting of pairs $(S^3,K).$  
 \begin{figure}
\labellist
 \pinlabel {\Large $-j$}  at 117 44
\endlabellist
\includegraphics[scale=0.85]{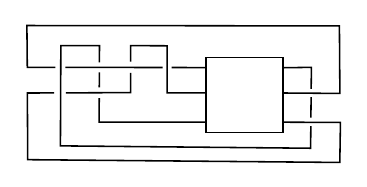}
\caption{The knot diagram of $K_j$ for $j\geq 0$. Here $-j$ indicates $j$ left-handed full twists.  }
\label{fig: kjknot}
\end{figure}

\begin{restatable}{proposition}{propkj} \label{prop: kj}
    The family $\{K_j\}_{j\geq 0}$ satisfies the following properties.
\begin{enumerate}
    \item \label{it:kj1} The family $\{(S^3_{+1}(K_j)\conn -S^3_{+1}(K_j),\mu \conn U)\}_{j>0}$ generates a $\Z^\infty$ summand in $\CZhat/\CZ$, where $\mu$ is the image of a meridian of $K_j$ and $U$ is the unknot;
    \item \label{it:kj2} $K_j$ is not homology concordant to any L-space knot. 
\end{enumerate}
\end{restatable}
Previous examples of families of knots that generate $\Z^\infty$ summands of the quotient group $\CZhat/\CZ$ use L-space knots in the place of $K_j$, see \cite[Corollary 1.2, Lemma 11.3]{dai2021homology}.  

This paper is organised as follows: in Section~\ref{sec:Bakground} we review relevant aspects of Heegaard Floer homology and $(1,1)$ knots while in Section~\ref{sec:mainthm} we prove Theorem~\ref{thm: main} and Proposition~\ref{prop: kj}. In Section \ref{sec: addendum} we give an algorithm for checking if a $(1,1)$ diagram is strongly almost coherent.

\subsection*{Acknowledgments}
The first author would like to thank Tao Li and Joshua Greene for some helpful conversations.
Some of this work was carried out  while the authors attended the ``Floer Homotopical Methods in Low Dimensional and Symplectic Topology" workshop at the Simons
Laufer Mathematical Sciences Institute (formerly MSRI) supported by NSF Grant DMS-1928930\ when the second author was in residence there during Fall 2022. Further work was carried out while the authors attended the ``2022 Tech topology conference". We would like to thank the organisers of these two conferences. The second author was partially supported by NSF grant DMS-2104144 and a Simons
Fellowship. Finally, the authors would like to express their gratitude to an anonymous referee for their thorough reading and exceptionally detailed report, which has led to significant improvements in both the mathematical rigor and the exposition of the paper. 

\section{Background and Review}\label{sec:Bakground}

In this section we review aspects of Heegaard Floer homology and $(1,1)$ knots that will be relevant in Section~\ref{sec:mainthm}.

\subsection{Knot Floer homology and almost L-space knots}
Knot Floer homology is a knot invariant due independently to Ozsv\'{a}th-Szab\'{o} \cite{OSknot} and J. Rasmussen \cite{Ras_complement}. We briefly review the definition.

For a knot $K$, in a rational homology three sphere $Y$, one associates a doubly-pointed Heegaard diagram $\hH=\{ \Sigma, \bm{\alpha}, \bm{\beta},z,w \} $, where $\Sigma$ is a genus $g$ surface. The symmetric product of $\Sigma$, \ $\Sym^g(\Sigma)$, is a $g$-dimensional {complex} manifold obtained by taking the quotient of the $g$-fold Cartesian product by the natural action of the symmetric group $S_g$. $\Sym^g(\Sigma)$ contains two {totally real tori} given as the image of the following manifolds
\[
\mathbb{T}_\alpha = \alpha_1 \times \cdots \alpha_g \hspace{2em} \text{and} \hspace{2em} \mathbb{T}_\beta = \beta_1 \times \cdots \beta_g.
\] under the quotient. Ozsv\'{a}th-Szab\'{o} define a map 
\[
\spincs_w \colon \mathbb{T}_\alpha \cap \mathbb{T}_\beta \rightarrow \Spin^c(Y)
\]
in \cite[Section 2.6]{OSclosed}. 
Set $\mathfrak{T}(\hH, s) = \{ x \in \mathbb{T}_\alpha \cap \mathbb{T}_\beta \ | \  \spincs_w(x) = s \}.$ Suppose $[K]$ is of order $d$ in $H_1(Y),$ then there is a function 
\[
A_{w,z} \colon \mathfrak{T}(\hH, s) \rightarrow \frac{1}{2d}\Z
\]
called the \emph{Alexander grading.} See \cite[Section 2.2]{HL}, \cite[Section 2.4]{rational}. 
For each $s \in \Spin^c(Y)$, the knot Floer chain complex $\CFKi(Y,K,s)$ is generated by the tuples $[x,i,j]$ where $x\in \mathfrak{T}(\hH, s)$, $i\in \Z$ and $j-i = A_{w,z}(x).$ Note that for generators in the same spin$^c$ structure, all $j$ are in the same coset of $\Z.$   This complex is often viewed as a free module over the polynomial ring $\F[U,U^{-1}],$ where $\F = \Z/2\Z$ and the action of $U$ is given by $U\cdot [x,i,j] = [x, i-1, j-1]$. The differential is given by
\[
\partial x = \sum_{y\in \mathbb{T}_\alpha \cap \mathbb{T}_\beta} \sum_{\substack{\phi \in \pi_2(x,y)\\ \mu(\phi)=1}} \# \widehat{\mathcal{M}}(\phi)[y,i-n_w(\phi),j-n_z(\phi)],
\]
where $\pi_2(x,y)$ is the set of homotopy class of Whitney disks between $x$ and $y,$ $\widehat{\mathcal{M}}(\phi)$ is the reduced moduli space of the holomorphic representatives of $\phi$,   $\mu$ is the Maslov index, and $n_w$ (resp.~$n_z$) is the intersection number of $\phi$ with a co-dimension $2$ submanifold of $\Sym^g(\Sigma)$ associated to the $w$ (resp.~$z$) basepoint. Let 
\[\CFKi(Y,K) = \bigoplus_{s \in \Spin^c(Y)} \CFKi(Y,K,s).\]
There is a double-filtration on $\CFKi(Y,K)$ induced by $(i, j)$. 
It turns out that up to doubly-filtered chain homotopy equivalence, $\CFKi(Y,K)$ is a topological invariant of the pair $(Y,K).$ Additionally, each generator is assigned an integer value called the \emph{Maslov grading}, denoted by $\gr$. The differential lowers the Maslov grading by $1$ while the $U$ action lowers the it by $2.$ 

A \emph{positive (resp.~negative) L-space knot} is a knot $K$ in an L-space $Y$ such that sufficiently postive (resp.~negative) surgeries on $K$ yields an L-space.  It turns out that those knots have particularly simple knot Floer chain complex that resembles a staircase on the $(i,j)$ plane. 
\begin{definition} \label{def: staircase}
In $ \CFKi(Y,K,s),$ \emph{a positive staircase {of length $k$}}  is a set of generators $\{ x_\ell\}_{1\leq \ell \leq k} \cup \{ y_\ell\}_{1\leq \ell \leq k+1}$ where $k$ is a non-negative integer, such that $x_\ell$ and $y_\ell$ only differ in $i$ coordinate while $x_\ell$ and $y_{\ell+1}$ only differ in $j$ coordinate, and the {non-trivial components of the differential} are given by $\partial x_\ell = y_\ell + y_{\ell+1}    $ for $  1\leq \ell \leq k.  $ A \emph{negative staircase {of length $k$}} is the dual of a positive staircase {of length $k$}.
\end{definition}

Suppose $K\subset Y$ is a positive (resp.~negative) L-space knot. For each $s\in \Spin^c(Y),$  $\CFKi(Y,K,s)$ consists of a  positive (resp.~negative) staircase. 
For knots in $S^3$ this result follows from a result of Ozsv\'{a}th-Szab\'{o} \cite[Theorem 1.2]{OSlspace} while the general case is due to J. and S. Rasmussen \cite[Lemma 3.2]{JSRas}, building on work of Boileau, Boyer, Cebanu, and Walsh~\cite{boileau2012knot}. 

Next we turn to \emph{almost L-space knots}.  An \emph{almost L-space} is a rational homology sphere $Y$ such that $\HFh(Y,s)$ has rank $1$ in all but one spin$^c$ structure, in which it has rank $3$. We call the $\spin^c$ structure in which $\rank(\HFh(Y,s))=3$ the \emph{exceptional $\spin^c$ structure}. A knot $K$ in a rational homology sphere L-space is  a \emph{positive (resp.~negative) almost L-space knot} if sufficiently positive (resp.~negative) surgeries on $K$ yield almost L-spaces and $K$ is not a trefoil.

\begin{definition} \label{def:almoststaircase}
   {Let $k$ be a positive integer. A set of generators of $ \CFKi(Y,K,s),$} is called a \emph{positive almost staircase {of length $k$}}  if it is one of the following.
   \begin{enumerate} [label=(\alph*)]
       \item  {When $k$ is even,} generators are $\{ x_\ell\}_{1\leq \ell \leq {k}} \cup \{ y_\ell\}_{1\leq \ell \leq {k+1}} \cup \{y'_{{\frac{k}{2}+1}} \}  \cup \{z\},$ where $y_{{\frac{k}{2}+1}}, y'_{{\frac{k}{2}+1}}$ and $z$ are in coordinates $(0,1), (1,0)$ and $(0,0)$ respectively,  $x_\ell$ and $y_\ell$ only differ in $i$ coordinate, while  $x_\ell$ and $y_{\ell+1}$ only differ in $j$ coordinate except for $\ell= {\frac{k}{2}}$, where  $x_{{\frac{k}{2}}}$ and $y'_{{\frac{k}{2}+1}}$ only differ in $j$ coordinate. The {non-trivial components of the} differential {are} given by 
       \begin{align*}
           \partial x_\ell &= y_\ell + y_{\ell + 1}  \hspace{5em}    1\leq \ell \leq {k} \quad\text{and}\quad \ell \neq {\frac{k}{2}, \frac{k}{2}+1}  \\
           \partial x_{{\frac{k}{2}}} &= y_k + y_{{\frac{k}{2} + 1}} + y'_{{\frac{k}{2} + 1}}    \\
           \partial x_{{\frac{k}{2} + 1}} &= y_{{\frac{k}{2} + 1}} + y_{{\frac{k}{2} + 2}} + y'_{{\frac{k}{2} + 1}}    \\
           \partial y_{{\frac{k}{2} + 1}}  &= \partial y'_{{\frac{k}{2} + 1}} = z;
       \end{align*}
           \item    {When $k$ is odd,} generators are $\{ x_\ell\}_{1\leq \ell \leq {k} }  \cup \{x'_{{\frac{k+1}{2}}} \} \cup \{ y_\ell\}_{1\leq \ell \leq {k+1}}  \cup \{z\},$ where $x_{{\frac{k+1}{2}}}, x'_{{\frac{k+1}{2}}}$ and $z$ are at coordinate $(0,-1), (-1,0)$ and $(0,0)$ respectively,  $x_\ell$ and $y_\ell$ only differ in $i$ coordinate, while  $x_\ell$ and $y_{\ell+1}$ only differ in $j$ coordinate except for $\ell= {{\frac{k+1}{2}}}$,  $x'_{{\frac{k+1}{2}}}$ and $y_{{{\frac{k+1}{2}+1}}}$ only differ in $j$ coordinate. The {non-trivial components of the} differential {are} given by 
       \begin{align*}
           \partial x_\ell &= y_\ell + y_{\ell + 1}  \hspace{5em}    1\leq \ell \leq {k} \quad\text{and}\quad \ell \neq {\frac{k+1}{2}} \\
              \partial z &= x_{{\frac{k+1}{2}}} + x'_{{\frac{k+1}{2}}}\\
           \partial x_{{\frac{k+1}{2}}} &= \partial x'_{{\frac{k+1}{2}}} =  y_{{\frac{k+1}{2}}}  +  y_{{\frac{k+1}{2}+1}}.       
       \end{align*}
   \end{enumerate}
 {Roughly speaking, an almost staircase is obtained from a staircase of the same length by replacing the middle generator by three generators.}
 
   We note there are no known examples of knots with the $\CFK^\infty$-type of almost staircases of {odd length}.
A \emph{negative almost staircase} is the dual of a positive almost staircase.

A \emph{box with length 1} is a set of generators $\{a,b,c,d\}$ with coordinates $(0,0), (0,1), (1,0)$ and $(1,1)$ respectively and differential with non-trivial components given by 
\[
  \partial d = b + c \hspace{5em}    \partial b = \partial c = a.
\] 
\end{definition}

\begin{theorem}[~\cite{binns2023cfk}] \label{thm:almostlspace}
    Suppose $K$ is a positive almost L-space knot in a rational homology sphere. Then in the exceptional spin$^c$ structure $s$ $ \CFKi(Y,K,s)$ consists of either 
    \begin{itemize}
        \item a direct sum of a positive staircase and a box of length $1$
        \item or a positive almost staircase,
    \end{itemize}
    while in every other spin$^c$ structure the knot Floer chain complex consists of a positive staircase. 
\end{theorem}

Note that only almost L-space knots in $S^3$ are considered in~\cite{binns2023cfk}, but given that Ozsv\'ath-Szab\'o's large surgery formula~\cite[Section 4]{OSknot} holds for rational homology spheres, the result above follows from the proof in the $S^3$ case, together with Ozsv\'ath-Szab\'o's classification of the $\CFK^\infty$-type of L-space knots~\cite{OSlspace}. Note also that the statement holds for the negative almost L-space knot as well, replacing the word ``positive" by ``negative". Finally, we note in passing that the exceptional $s$ is necessarily self-conjugate.
\begin{figure}
\begin{minipage}{.5\linewidth}
     \centering
      \subfloat[]{
       \begin{tikzpicture}[scale=0.9]
       \begin{scope}[thin, black!0!white]
	  \draw  (5, 0) -- (-3,0);
   \draw  (0, 3) -- (0,-1);
      \end{scope}
      \begin{scope}[thin, black!40!white]  
     \draw  (0, -1) -- (0, 3);
	  \draw  (-1, 0) -- (3, 0);
      \end{scope}
      
      	\filldraw (0, 0) circle (2pt) node[] {};
       \filldraw (1, 0) circle (2pt) node[] {};
       \filldraw (0, 1) circle (2pt) node[] {};
       \filldraw (1, 2) circle (2pt) node[] {};
       \filldraw (2, 1) circle (2pt) node[] {};
       \filldraw (0, 2) circle (2pt) node[] {};
       \filldraw (2, 0) circle (2pt) node[] {};

     \draw  (0,0) -- (0,1);
     \draw  (0,0) -- (1,0);
     \draw  (1,0) -- (1,2);
          \draw  (0,2) -- (1,2);
      \draw  (0,1) -- (2,1);
       \draw  (2,1) -- (2,0);

       \draw  (2,1) -- (1,0);
        \draw  (1,2) -- (0,1);
  
      	 	\node  [left] at (0,-0.3) {$z$};
         \node  [left] at (1.2,-0.3) {$y'_2$};
           \node  [left] at (2.4,-0.3) {$y_3$};

    \node  [left] at (0,1) {$y_2$};
     \node  [left] at (0,2.2) {$y_1$};
      \node  [above] at (1,2) {$x_1$};
       \node  [right] at (2,1) {$x_2$};
    \end{tikzpicture}
     }
\end{minipage}%
\begin{minipage}{.5\linewidth}
     \centering
      \subfloat[]{
       \begin{tikzpicture}[scale=0.9]
       \begin{scope}[thin, black!0!white]
	  \draw  (-5, 0) -- (2,0);
      \end{scope}
      \begin{scope}[thin, black!40!white]  
     \draw  (0, -4) -- (0, 1);
	  \draw  (-4, 0) -- (1, 0);
      \end{scope}
      
      	\filldraw (0, 0) circle (2pt) node[] {};
       \filldraw (-1, 0) circle (2pt) node[] {};
       \filldraw (0, -1) circle (2pt) node[] {};
       \filldraw (-1, -2) circle (2pt) node[] {};
       \filldraw (-2, -1) circle (2pt) node[] {};
       \filldraw (1, -2) circle (2pt) node[] {};
       \filldraw (-2, 1) circle (2pt) node[] {};
          \filldraw (1, -3) circle (2pt) node[] {};
       \filldraw (-3, 1) circle (2pt) node[] {};

     \draw  (0,0) -- (0,-1);
     \draw  (0,0) -- (-1,0);
     \draw  (-1,0) -- (-1,-2);
          \draw  (1,-2) -- (-1,-2);
      \draw  (0,-1) -- (-2,-1);
       \draw  (-2,-1) -- (-2,1);

       \draw  (-2,-1) -- (-1,0);
        \draw  (-1,-2) -- (0,-1);

          \draw  (1,-2) -- (1,-3);
                    \draw  (-2,1) -- (-3,1);
  
      	 	\node  [right] at (0,0.3) {$z$};
         \node  [right] at (-1.2,0.3) {$x'_2$};
           \node  at (-1.7,0.7) {$x_1$};

    \node  [right] at (0,-1) {$x_2$};
     \node  [right] at (0.4,-1.8) {$x_3$};
      \node  [below] at (-1,-2) {$y_3$};
       \node  [left] at (-2,-1) {$y_2$};
       \node  [above] at (-2.8,0.5) {$y_1$};
              \node  [right] at (0.4,-2.8) {$y_4$};
    \end{tikzpicture}
     }
\end{minipage}
    \caption{Two examples of positive almost staircases drawn in the $(i,j)$ plane. {The almost staircase in (a) is of length $2$ and the staircase in (b) is of length $3$}. Solid dots indicate generators and edges indicate non-trivial components of the differential (since the differential lowers the filtration level we omit the direction of the edges).}
    \label{fig: almost_staircases}
\end{figure}
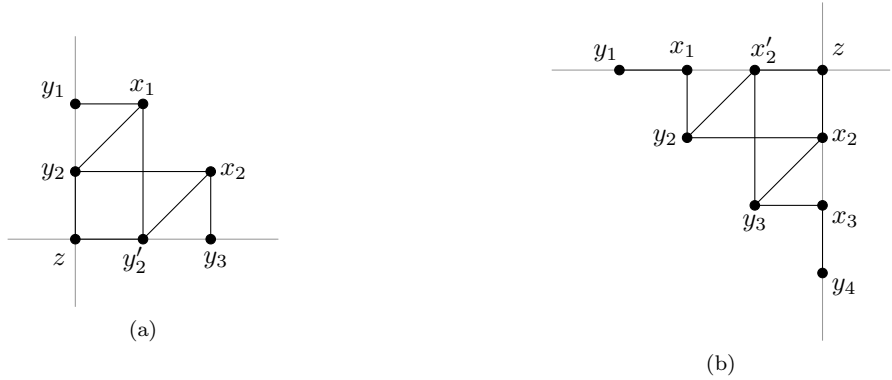

\subsection{(1,1) knots} \label{subsec: 11}

 A $(1,1)$ knot is a knot which can be encoded by a doubly pointed Heegaard diagram on $S^1\times S^1$. {In this paper we only consider $(1,1)$-knots in lens spaces --- i.e. not in $S^1\times S^2$ --- since Theorem~\ref{thm:almostlspace} only applies to knots in rational homology spheres.} For $(1,1)$ knots, the definition of the knot Floer chain complex simplifies drastically. Given $\hH=\{ \Sigma, \alpha, \beta,z,w \}$,  where the genus of $\Sigma$ is $1$, we have that $\Sym^g(\Sigma)=\Sigma, \mathbb{T}_\alpha=\alpha$ and $\mathbb{T}_\beta=\beta$. Consider the universal cover of $\Sigma$, $\R^2$. The orientation on $\Sigma$ induces an orientation on $\R^2$. Let $\tilde{\alpha},\tilde{\beta}$ denote particular lifts of $\alpha$ and $\beta$ respectively.  Orienting $\alpha$ and $\beta$ so that $\alpha \cdot \beta >0$; this induces orientations on $\tilde{\alpha}$ and $\tilde{\beta}$ so that $\tilde{\alpha} \cdot\tilde{\beta} >0$. For $x,y \in \mathfrak{T}(\hH) = \alpha \cap \beta$, we have $\langle \partial x,y\rangle\neq 0$ if and only if there is an embedded bigon from a lift of $x$ to a lift of $y$ cobounded by $\tilde{\alpha}$ and $\tilde{\beta}$ (subject to the standard orientation conditions). In sum, computing Heegaard Floer homology in the case of $(1,1)$ knots requires no analysis.

 Up to isotopy, we can ensure that each bigon in a $(1,1)$ diagram contains at least one basepoint by successively isotoping away bigons in the complement of the basepoints, that is, by simply ``pulling tight'' the curves. Then by fixing a choice of the parametrization of $S^1\times S^1$, each $(1,1)$ diagram admits a unique \emph{standard diagram}, encoded by a four-tuple $(p,q,r,s)$ following the convention of J. Rasmussen~\cite{Ras_homology}.
See Figure \ref{fig: standard}. The number $p$ is an invariant of the knot {--- it is the rank of the knot Floer homology of the knot --- while $q$ depends on the diagram}. A priori, $s$ may take value in $\Z$. However, note that the transformation $(p,q,r,s) \rightarrow (p,q,r,s+p)$ amounts to performing a Dehn twist along the red $\alpha$ curve, which changes the isotopy class of the diagram but preserves the manifold-knot pair. Therefore for each $(1,1)$ diagram we can choose the unique standard diagram representative such that $s$ is in the interval $[0,p).$ 
There is an involution given by reflection in the horizontal direction, which corresponds to taking the mirror of the manifold-knot pair, and corresponds to the transformation $(p,q,{r},{s}) \rightarrow (p,q,p-2q-r,2q-s).$  Each $(1,1)$ diagram is uniquely associated to a standard diagram and therefore to a four-tuple, but different standard diagrams can give rise to the same knot. We denote the $(1,1)$ knot associated to a four-tuple $(p,q,{r},{s})$ by $K(p,q,r,s).$

\begin{remark}
In the convention of \cite{GLV}, a $(1,1)$ diagram is called \emph{reduced} if each bigon contains at least one basepoint. Note that the difference between a reduced $(1,1)$ diagram and a standard diagram is merely whether one fixes a parametrization of $S^1\times S^1$ or not. Therefore one does not lose generality by making statements regarding standard diagrams.
\end{remark}

We will be interested in standard diagrams that satisfy various properties. First recall the following definition from the introduction:
\coherent*
L-space knots are diagrammatically characterized amongst $(1,1)$ knots as follows:
\lspace*
We now give a reinterpretation of Greene-Lewallen-Vafaee's proof of Theorem \ref{thm: lspace}. This will give an indication of how we shall proceed in the case of almost L-space knots.

Fix a lift $\wa$ of $\alpha$ in the universal cover $\R^2$. 
Given  $s \in \Spin^c(Y),$ choose a lift of some point $x\in \mathfrak{T}(\hH,s)$ to a point $\tilde{x}\in\wa.$ There is a unique lift $\wb_s$ of $\beta$ that passes through {$\tilde{x}$}.  {Up to rigid motion of the plane, $\wa \cup \wb_s$ does not depend on the choice of $\wa, x$ or $\tilde{x}$, but only on the spin$^c$ structure $s$.}
The points of $\wa \cap \wb_s$ are in one-to-one correspondence with $\mathfrak{T}(\hH,s)$. There are $2n+1$ such points for some non-negative integer $n$ since $\chi(\widehat{CFK}(\hH,s))=1$. 
Number these points $\wa \cap \wb_s$ from $1$ to $2n+1$ following the orientation of $\wa$. Next, starting from infinity and following the orientation of $\wb_s$, record the sequence  of the intersection points that  $\wb_s$ passes through.
\begin{figure}
       \begin{tikzpicture}[scale=1]
       \begin{scope}[thin, red!50!white]
	  \draw  (0, 0) -- (4,0);
      \end{scope}    
      \filldraw (1, 0) circle (2pt) node[] {};
      \filldraw (2, 0) circle (2pt) node[] {};
      \filldraw (3, 0) circle (2pt) node[] {};
       \draw[arc arrow=to pos 0.7 with length 2mm, blue] (2,0) arc (180:0:0.5);     
           \draw[arc arrow=to pos 0.7 with length 2mm, blue] (1,0) arc (-180:0:1);   
    \end{tikzpicture}
\caption{The right most intersection point is a turning point.}\label{fig: for_def_turning}
\end{figure}
\begin{definition}\label{def:turningpoint}
An intersection point $b\in \wa \cap \wb_s$ is called a \emph{turning point} if its number is either greater or less than the numbers of both its successor and predecessor in the above sequence. See Figure \ref{fig: for_def_turning} for an example.
\end{definition}
{We observe the following:
\begin{lemma}\label{lem:turningpointarrowsgradings}
   Viewed as a basis element for $\CFKi(Y,K,s)$, a turning point admits both an incoming and an outgoing arrow. In particular, the basis $\mathfrak{T}(\mathcal{H},s)$ contains generators in three consecutive Maslov gradings.
\end{lemma}
\begin{proof}
    The two bigons connecting the turning point, $t$, to its successor, $p$, and predecessor, $s$,  --- as shown in Figure~\ref{fig: for_def_turning} --- are of opposite orientations. Therefore one of them induces an incoming arrow and the other one induces an outgoing arrow. The condition on the Maslov gradings follows from the fact that one of $p$, $s$ has Maslov grading one larger than that of $t$, while the other has Maslov grading one lower than that of $t$.
\end{proof}}
 However, {there is no generator with an incoming and outgoing arrow} in the {reduced} knot Floer chain complex of an L-space knot for any basis. We note in passing that in the knot Floer chain complex of an L-space knot there is in fact a unique basis that can be obtained from a $(1,1)$ diagram.
The following fact is implicit in \cite{GLV}, {and follows directly from the previous observation and Lemma:}
\begin{lemma}[\cite{GLV}]\label{le:lspace_noturning}
 A standard diagram for an L-space knot has no turning point in the universal cover. 
\end{lemma}
In other words, the order of the sequence of the intersection points along $\wb_s$ is either increasing or decreasing. {Curves with this property \change{are} called \emph{graphic} in \cite{GLV}.
\begin{definition}(\cite[Definition 2.5]{GLV})
    An oriented curve (or arc, or ray) $\wb \subset \R^2$ is called positive (resp.~negative) graphic if its intersection points with $\wa$ occur in the same (resp.~opposite) order along $\wa$ and $\wb$.
\end{definition}
}
 Whether $\wb_s$ is positive or negative graphic depends only on whether the L-space knot is positive or negative. Therefore if this property holds for one $\spin^c$ structure it holds for all spin$^c$ structures. It follows that the $(1,1)$ diagram is coherent. More precisely, $\wa$  separates $\R^2$ into upper and lower half planes, and in each half plane it cuts $\wb_s$ into $n$ arcs and $1$ ray that goes to infinity.  Each rainbow arc in the $(1,1)$ diagram around, say, the $z$ basepoint has a lift to an arc in the lower half plane. Since the condition in Lemma \ref{le:lspace_noturning} implies that the arcs in the lower half plane have the same orientation,  the rainbow arcs themselves have the same orientation.

The converse of Lemma \ref{le:lspace_noturning} is also true, as follows. {
\begin{lemma}(\cite[Proposition 2.6]{GLV})\label{le:graphic_implies_staircase}
    Given a standard $(1,1)$-diagram, if the lift $\wb_s$ is positive (resp.~negative) graphic, then it represents a positive (resp.~negative) staircase.
\end{lemma}}
This is proved by using the fact that the vertical homology and the horizontal homology are both one-dimensional for $\CFKi(Y,K,s).$

\section{Proof of the main theorem}\label{sec:mainthm}

In this section we prove Theorem~\ref{thm: main} by adapting the techniques outlined in Subsection~\ref{subsec: 11}. We also prove Proposition~\ref{prop: kj}.



Observe that if a standard $(1,1)$ diagram has only one rainbow arc then it is necessarily coherent, in which case it is the diagram for an L-space knot by the main result of Greene-Vafaee-Lewallen~\cite{GLV}. We can thus restrict our attention to $(1,1)$ diagrams which have at least two rainbow arcs.

Let $(S^1\times S^1,\alpha,\beta,z,w)$ be a $(1,1)$ diagram. As in the setting of L-space knots, we consider the lifts of the $\alpha$ and $\beta$ curves to the universal cover of $S^1 \times S^1$, which is $\R^2$, {where the covering map is the usual quotient $\R^2\surj\R^2/\Z^2$ for which $\Z\oplus \Z$ acts by addition on $\R^2$.} Recall that for a reduced $(1,1)$ diagram, in the universal cover given a fixed lift $\wa$, we can choose a lift  $\wb_s$ for each $s\in \Spin^c(Y)$. The curve $\wb_s$ is cut by $\wa$ into arcs and $2$ rays. 
Without loss of generality we may take $\tilde{\alpha}$ to be the line $\{y=0\} \subset \R^2$.

Define $h(\eta)$, the \emph{height} of an arc $\eta$  as follows:
\begin{align}h(\eta) =\max
\{\lvert y\rvert \colon(x,y)\in \eta\}.
\end{align}

We also define a notion of inconsistent arc in the universal cover $\R^2$, that mirrors the definition in $S^1\times S^1$. Consider the collection of curves $\{\tilde{\beta}_s:s\in \Spin^c(Y)\}$. If there are two orientations amongst all the arcs in each half plane, call the direction in the minority the \emph{inconsistent direction}, and the arcs in this direction the \emph{inconsistent arcs}. Call the remaining arcs \emph{consistent arcs} and their direction the \emph{consistent direction}.

\begin{definition} \label{def: vir_almost}
  A standard diagram $(S^1\times S^1,\alpha,\beta,z,w)$ is \emph{virtually almost coherent} if $\underset{s\in \Spin^c(Y)}{\bigcup}\tilde{\beta}_s$ contains exactly two inconsistent arcs.
\end{definition}

Note that if there are two inconsistent arcs, there must be one in each half plane. Clearly any standard virtually almost coherent $(1,1)$ diagram has exactly two inconsistent arcs, since each rainbow arc around $z$~(resp. $w$) basepoint lifts to at least one arc in the lower~(resp. upper) half plane. The converse is false, however, as illustrated by the example $K(9,2,0,4)$ in Figure \ref{fig: K9204}.

\begin{figure}
\begin{minipage}{.5\linewidth}
\centering
\subfloat[The standard diagram of $K(9,2,0,4).$ The $i$-th intersection point on the top side is identified with the $(i+4)$-th point on the bottom side.]{
\labellist
 \pinlabel { $w$}  at 58 25
 \pinlabel { $z$}  at 60 160
\endlabellist
\includegraphics[scale=0.85]{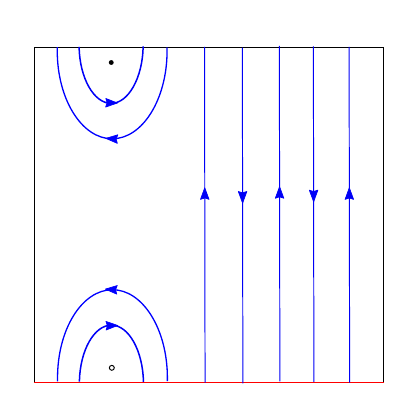}
}
\end{minipage}%
\begin{minipage}{.5\linewidth}
\centering
\subfloat[The lift to the universal cover of $K(9,2,0,4)$. Thick solid dots indicate a choice of basis for the knot Floer complex.]{
\includegraphics[scale=0.9]{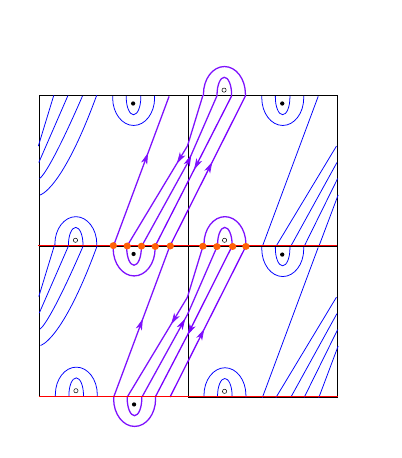}}
\end{minipage}\\
\begin{minipage}{.5\linewidth}
     \centering
      \subfloat[The lift to the universal cover.]{
       \begin{tikzpicture}[scale=0.8]
       \begin{scope}[thin, red!50!white]
	  \draw  (-1, 0) -- (9,0);
      \end{scope}
     \begin{scope}[thin, black!0!white]
          \draw  (0, -2) -- (0, 3);
      \end{scope}
      
   \foreach \i in {0,...,8}
   {
      	\filldraw (\i, 0) circle (2pt) node[] {};
       }
        \draw[new arc arrow=to pos 0.7 with length 2mm, blue] (3,0) arc (0:-180:1.5);
     \draw[new arc arrow=to pos 0.7 with length 2mm, blue] (1,0) arc (-180:0:0.5);
       \draw[new arc arrow=to pos 0.7 with length 2mm, blue] (2,0) arc (180:0:0.5);
         \draw[new arc arrow=to pos 0.7 with length 2mm, blue] (4,0) arc (0:180:1.5);
           \draw[new arc arrow=to pos 0.7 with length 2mm, blue] (7,0) arc (0:-180:1.5);
     \draw[new arc arrow=to pos 0.7 with length 2mm, blue] (5,0) arc (-180:0:0.5);
       \draw[new arc arrow=to pos 0.7 with length 2mm, blue] (6,0) arc (180:0:0.5);
         \draw[new arc arrow=to pos 0.7 with length 2mm, blue] (8,0) arc (0:180:1.5);
         
              \node   at (2.5,0.2) {\small $w$};
          \node   at (1.5,-0.2) {$z$};
                  \node   at (6.5,0.2) {\small $w$};
          \node   at (5.5,-0.2) {$z$};
  \draw [blue](0,0) -- (0,2);
  \draw [blue](8,0) -- (8,-1.6);
         
    \end{tikzpicture}\label{subfig: k9204_c}
     }
\end{minipage}%
\begin{minipage}{.5\linewidth}
     \centering
      \subfloat[The knot Floer complex.]{
       \begin{tikzpicture}[scale=0.9]
        \begin{scope}[thin, black!0!white]
          \draw  (-5, 0) -- (5, 0);
            \draw  (0, 2.5) -- (0, -2);
      \end{scope}
      \begin{scope}[thin, black!40!white]  
      \foreach \i in {-1,...,1}
   {
   \draw  (-1.75, \i) -- (1.75, \i);
     \draw  (\i, -1.75) -- (\i, 1.75);
       }
      \end{scope}
      \filldraw (0.8, 0.8) circle (2pt) node[] (){};
      	\filldraw (0.6, 0.6) circle (2pt) node[] (){};
       \filldraw (0.6, -0.6) circle (2pt) node[] (){};
       \filldraw (-0.6, 0.6) circle (2pt) node[] (){};
              \filldraw (-0.6, -0.6) circle (2pt) node[] (){};
              \filldraw (-0.4, -0.4) circle (2pt) node[] (){};
                \filldraw (0.4, -0.4) circle (2pt) node[] (){};
                  \filldraw (-0.4, 0.4) circle (2pt) node[] (){};
                    \filldraw (0.4, 0.4) circle (2pt) node[] (){};

     \draw  (0.6, 0.6) -- (0.6, -0.6);
          \draw  (0.6, 0.6) -- (-0.6, 0.6);
        \draw  (0.6, -0.6) -- (-0.6, -0.6);
             \draw  (-0.6, 0.6) -- (-0.6, -0.6);
                  \draw  (-0.4, -0.4) -- (0.4, -0.4);
                     \draw  (-0.4, -0.4) -- (-0.4, 0.4);
                        \draw  (0.4, 0.4) -- (0.4, -0.4);
                           \draw  (0.4, 0.4) -- (-0.4, 0.4);

    \end{tikzpicture} \label{subfig: K9204_knotfloer}
     }
\end{minipage}
\caption{The knot $K(9,2,0,4)$ has exactly two inconsistent arcs downstairs, but these lift to four inconsistent arcs in the universal cover. Similar examples abound amongst $(1,1)$ knots.}
\label{fig: K9204}
\end{figure}

It turns out that virtually almost coherent is equivalent to strongly almost coherent. We recall the definition of strongly almost coherence for the reader's convenience.
\strongly*

In order to show that the two versions of almost coherence are equivalent, we start with a geometric observation. 

\begin{lemma}\label{le: onebasepoint}
Suppose $\mathcal{H}$ is a virtually almost coherent diagram. Then the height of the inconsistent arc is less than $1$.
\end{lemma}

In particular, if $\gamma$ is an inconsistent arc in a virtually almost coherent diagram, the bigon formed by $\gamma$ and $\wa$ contains exactly one basepoint.

\begin{proof}
    Suppose  otherwise. Then there exists an inconsistent arc $\gamma$ in the upper half plane  such that $h(\gamma)>1$. Consider any bigon $D$ formed by $\gamma$ and the line $\{(x,y)\in\R^2:y=\lfloor h(\gamma)\rfloor\}$. A translate of $D$ gives another inconsistent arc of a smaller height, contradicting the virtually almost coherent condition.
\end{proof}

\begin{proposition} \label{pro: equiv_def}
    A standard diagram is virtually almost coherent if and only if it is strongly almost coherent.
\end{proposition}
\begin{proof}

 Let $\mathcal{H}=(\alpha,\beta,z,w)$ be a standard diagram. Suppose $\mathcal{H}$ is virtually almost coherent. Note that by Lemma~\ref{le: onebasepoint} we can assume that the inconsistent arcs are lifts of rainbow arcs. If the two inconsistent arcs in the universal cover share a common end point, the first case in Definition~\ref{def: strongly_almost_coherent} applies. Suppose the two inconsistent arcs do not share an end point, then up to reflection with respect to vertical and horizontal axis, the only possible positions for the two inconsistent arcs are depicted in Figure \ref{fig: virtualtostrong} \subref{subfig:vtos_1},\subref{subfig:vtos_2} and \subref{subfig:vtos_3}.
 
 {We first show that \subref{subfig:vtos_1} and \subref{subfig:vtos_2} cannot occur. Notice that in either case the portion of the lift $\wb_s$ not shown in the figure consists only of graphic rays. Assign any direction to the inconsistent arcs, in both cases \subref{subfig:vtos_1} and \subref{subfig:vtos_2}, the two graphic rays from the end points of either inconsistent arc either intersect each other or both go to infinity, a contradiction.}

{Suppose we are in case \subref{subfig:vtos_3}. Once again, the portion of the lift $\wb_s$ not shown in the figure consists only of graphic rays. The bigon formed by an inconsistent arc and $\wa$ contains at most one base-point. No matter in which direction the inconsistent arcs go, it follows that the only possible diagram is the one depicted in Figure \ref{fig: virtualtostrong}~\subref{subfig:vtos_4}, where the dashed curves indicates the remaining graphic rays. This is the second case in Definition~\ref{def: strongly_almost_coherent}.}
    
         Suppose now that $\mathcal{H}$ is strongly almost coherent. We want to show that the only inconsistent arcs in the universal cover are those of height less than $1$; in other words, the lifts of the two inconsistent rainbow arcs. Consider any other arc and the bigon $D$ it forms with $\wa.$ Any inconsistent sub-arcs on $\partial D$ appears in one of the two local models  in Definition~\ref{def: strongly_almost_coherent}; observe that such local inconsistency also necessarily occurs on  $\wa' \colon= \{ (x,y)\in \R^2 :y = c \}$ for some $c \neq 0.$ Replacing the entire sub-arc of $\wb$ in the local model by the sub-arc of $\wa'$ with the same end points preserves the orientation of $\partial D$.  After performing this operation on every local inconsistency on $\partial D$, the only local maximum or minimum are of consistent orientation. It follows that $\partial D$ is consistently oriented.       
    \end{proof}

 \begin{figure}[htb!]
 \begin{minipage}{.5\linewidth}
      \centering
       \subfloat[]{
      \begin{tikzpicture}[scale=0.9]
        \begin{scope}[thin, red!50!white]
 	  \draw  (-1, 0) -- (4,0);
       \end{scope}
    
    \foreach \i in {0,...,3}
    {
       	\filldraw (\i, 0) circle (2pt) node[] {};
        }
       
       \draw [blue, dashed](2,0) -- (1.9,0.65);
       \draw [blue, dashed](1,0) -- (1.1,0.65);

            \draw[arc arrow=to pos 0.7 with length 2mm, blue] (3,0) arc (0:180:1.5);
            \draw[arc arrow=to pos 0.7 with length 2mm, blue] (2,0) arc (0:-180:0.5);
   
     \end{tikzpicture}\label{subfig:vtos_1}
      }
 \end{minipage}%
 \begin{minipage}{.5\linewidth}
      \centering
       \subfloat[]{
        \begin{tikzpicture}[scale=0.9]
        \begin{scope}[thin, red!50!white]
 	  \draw  (-1, 0) -- (5,0);
       \end{scope}
    
    \foreach \i in {-0.5,1.5,2.5,4.5}
    {
       	\filldraw (\i, 0) circle (2pt) node[] {};
        }
       
       \draw [blue, dashed](-0.5,0) -- (-0.4,0.65);
       \draw [blue, dashed](1.5,0) -- (1.4,0.65);
        \draw[arc arrow=to pos 0.7 with length 2mm, blue] (1.5,0) arc (0:-180:1);

            \draw[arc arrow=to pos 0.7 with length 2mm, blue] (2.5,0) arc (180:0:1);

     \end{tikzpicture}\label{subfig:vtos_2}
      }
 \end{minipage}\\
 \begin{minipage}{.5\linewidth}
      \centering
       \subfloat[]{
        \begin{tikzpicture}[scale=0.9]
        \begin{scope}[thin, red!50!white]
 	  \draw  (-1, 0) -- (6,0);
       \end{scope}
    
    \foreach \i in {0,1,3,4}
    {
       	\filldraw (\i, 0) circle (2pt) node[] {};
        }

            \draw[arc arrow=to pos 0.7 with length 2mm, blue] (3,0) arc (0:180:1.5);
            \draw[arc arrow=to pos 0.7 with length 2mm, blue] (4,0) arc (0:-180:1.5);
         
     \end{tikzpicture}\label{subfig:vtos_3}
      }
 \end{minipage}%
 \begin{minipage}{.5\linewidth}
      \centering
       \subfloat[]{
        \begin{tikzpicture}[scale=0.9]
        \begin{scope}[thin, red!50!white]
 	  \draw  (-1, 0) -- (6,0);
       \end{scope}
    
    \foreach \i in {0,...,4}
    {
       	\filldraw (\i, 0) circle (2pt) node[] {};
        }
       
       \draw [blue, dashed](4,0) -- (3.92,0.75);
       \draw [blue, dashed](0,0) -- (0.08,-0.75);
        \draw[arc arrow=to pos 0.7 with length 2mm, blue] (1,0) arc (180:0:0.5);
          \draw[arc arrow=to pos 0.7 with length 2mm, blue] (2,0) arc (-180:0:0.5);
            \draw[arc arrow=to pos 0.7 with length 2mm, blue] (3,0) arc (0:180:1.5);
            \draw[arc arrow=to pos 0.7 with length 2mm, blue] (4,0) arc (0:-180:1.5);
         
     \end{tikzpicture}\label{subfig:vtos_4}
      }
 \end{minipage}
     \caption{ Up to reflection with respect to vertical and horizontal axis, the only possible positions for the two inconsistent arcs are Figure \protect\subref{subfig:vtos_1} \protect\subref{subfig:vtos_2} and \protect\subref{subfig:vtos_3}. }
     \label{fig: virtualtostrong}
 \end{figure}
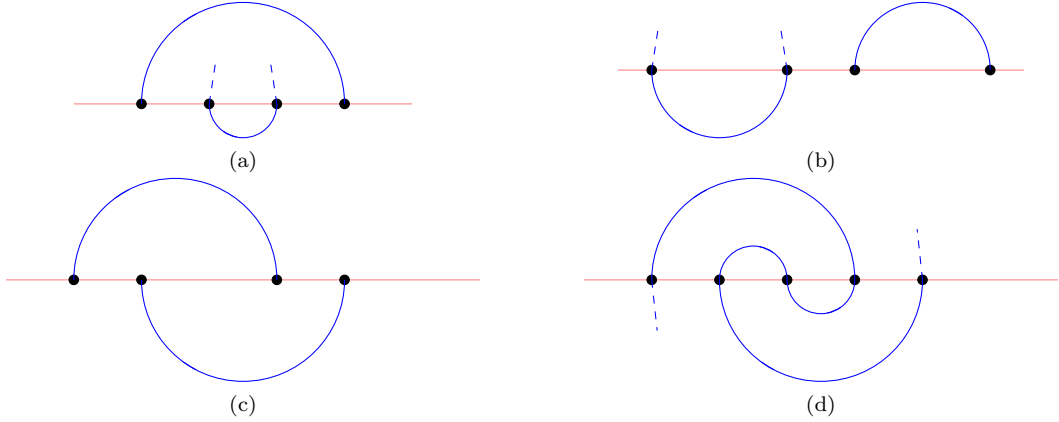

Thus to prove Theorem~\ref{thm: main} it suffices to prove the following.
\begin{theorem}\label{thm:(1,1)virtualboxcoherent}
    A standard diagram represents an almost L-space knot if and only if it is virtually almost coherent. Moreover, if $K\subset Y$ is  a $(1,1)$ almost L-space knot and $s$ is the exceptional spin$^c$ structure, then $\CFKi(Y,K,s)$  consists of the direct sum of a staircase and a box of length one.
\end{theorem}
We will prove both directions of this theorem individually in Proposition \ref{prop:neccesary} and \ref{prop:sufficient}. First we show that if $K$ is a  $(1,1)$ almost L-space knot then any standard diagram for $K$ is virtually almost coherent.

Up to translation by $U$ as a whole,  $\mathfrak{T}(\hH,s) = \wa \cap \wb_s$ gives rise to a specific basis of $\CFKi(Y,K,s)$ as follows. Choose any $x\in \mathfrak{T}(\hH,s)$. Starting from $[x,0,A_{w,z}(x)]$, {for each non-trivial component $[y,i,j]$ of $\partial  [x,0,A_{w,z}(x)]$, or each $[y,i,j]$ such that $[x,0,A_{w,z}(x)]$ has non-zero coefficient in  $\partial [y,i,j]$,} add $[y,i,j]$ into the basis. Continue this process until exhausting $\mathfrak{T}(\hH,s)$. The resulting basis is connected by components of the differential of the chain complex. Abusing notation we let $\mathfrak{T}(\hH,s)$ also denote this specific basis for $\CFKi(Y,K,s)$.

Throughout the remainder of this section we let $K$ be an almost L-space knot in a rational homology sphere $Y$ with exceptional spin$^c$ structure $s$. 
\begin{lemma} \label{le:threearrows}
Let $\hH$ be a standard diagram representing an almost L-space knot.
 The generators in the basis $\mathfrak{T}(\hH,s)$ are in either $3$ or $4$ consecutive Maslov gradings. Moreover, the number of the generators in each Maslov grading (from high to low) is given by one of the following ;
 \begin{enumerate}[label=(\alph*)]
     \item $(1,*,*)$,   \label{it:gen_num_1}
     \item $(*,*,1)$, \label{it:gen_num_2}
     \item $(1,2,*,*)$,\label{it:gen_num_3}
     \item $(*,*,2,1)$,\label{it:gen_num_4}
 \end{enumerate}
 where each $*$ indicates some (potentially different) positive integer. 
\end{lemma}
\begin{proof}

    According to Theorem \ref{thm:almostlspace},  up to filtered change of basis, $\CFKi(Y,K,s)$ is either a summand of a staircase and a length $1$ box  or an almost staircase. {Since the filtered change of basis preserves the sequence of the Maslov gradings, in each case we only need to check whether the condition holds for the standard basis.}

Start with the first case.  Suppose that after a filtered change of basis,  $\mathfrak{T}(\hH,s)$ becomes the set $ \mathfrak{S}_1 \cup \mathfrak{S}_2$ where $\mathfrak{S}_1$ generates {a length $k$ positive} staircase summand and  $\mathfrak{S}_2$ generates a box summand. 
{ For such a change of basis to be possible, $\mathfrak{S}_1$ and $\mathfrak{S}_2$ must coincide in at least $1$ Maslov grading. 
 Without the loss of generality, assume $\mathfrak{S}_1$ is supported in Maslov gradings $0$ and $1$. Thus the lowest Maslov grading of generators in $\mathfrak{S}_2$ can only be $-2,-1,0$ or $1$. Corresponding to these $4$ cases, the number of the generators in each Maslov grading (from high to low) is given by $(k,k+2,2,1), (k+1,k+3,1), (1,k+2,k+2)$ and $(1,2,k+1,k+1),$ respectively. This proves the lemma for the case of the positive staircase with a box summand.  Since taking the dual complex reverses the order of the sequence of the Maslov gradings, the case of negative staircases follows. } 

Suppose that after a filtered change of basis $\mathfrak{T}(\hH,s)$ generates a {length $k$ positive} almost staircase complex. {We check that for the standard basis in Definition~\ref{def:almoststaircase} the number of the generators in each Maslov grading is given by $(k,k+2,1)$ when $k$ is even, and $(1,k+1,k+1)$ when $k$ is odd.  The case of negative almost staircases follows from taking the dual complex.}
\end{proof}

We now refine the definition of ``turning point" given in Definition~\ref{def:turningpoint}. Note that by definition, a turning point necessarily has a predecessor and a successor (along $\wb_s$, following its orientation).
\begin{definition}
    Suppose $b$ is a turning point and $a$ its predecessor. If {$a$ has non-zero coefficient in  $\partial b$}, we call $b$ a \emph{positive} turning point and if {$b$ has non-zero coefficient in $\partial a$}, we call $b$ a \emph{negative} turning point.
\end{definition}

Note that a turning point together with its predecessor and successor live in $3$ consecutive Maslov gradings {by Lemma~\ref{lem:turningpointarrowsgradings}}. The order in which $\wb_s$ passes through {the three points} depends on the sign of the turning point. 

\begin{lemma} \label{le:maslov_shift}
    Moving along $\wb_s$, suppose $b$ and $c$ (in that order) are two closest turning points. Consider the set of the points in $\wa \cap \wb_s$ between $b$ and $c$ along $\wb_s$ (including the end points) and label them in order by $x_0 = b, x_1, \cdots, x_\ell, \cdots.$
  Their Maslov gradings satisfy \begin{align*}
     \gr(x_\ell) = 
      \begin{cases}
        \gr(b)  \qquad & \ell  \text{ is even}\\
        \gr(b) + 1   & \ell \text{ is odd and } b  \text{ is a positive turning point }\\
        \gr(b) - 1   & \ell \text{ is odd and } b  \text{ is a negative turning point}.
      \end{cases}
  \end{align*}
\end{lemma}
\begin{proof}
    Suppose that $b$ is a positive turning point. Then we have 
    \[
    \partial x_\ell   =  x_{\ell + 1} + x_{\ell - 1}
    \]
    for each odd $\ell$. If $b$ is a negative turning point, then we have 
    \[
    \partial x_{\ell + 1}  =  \partial x_{\ell - 1} = x_\ell
    \]
    for each odd $\ell$. The Maslov grading results follow.
\end{proof}




{Recall from Lemma~\ref{le:threearrows} that there is either a unique generator of $\mathfrak{T}(\mathcal{H},s)$ of maximal Maslov grading, or a unique generator of minimal Maslov grading.}
\begin{lemma} \label{le:presec}
    Let $\hH$ be a standard diagram representing an almost L-space knot. The unique generator in $\mathfrak{T}(\hH,s)$ with the highest/lowest Maslov grading has both a predecessor and a successor along $\wb_s.$
\end{lemma}
\begin{proof}   
    Up to reflecting $\R^2$ about $\wa$ 
    (recall that the mirror of an almost L-space knot is also an almost L-space knot)\change{,   we may assume} by Lemma~\ref{le:threearrows}  that the number of generators in each Maslov grading is either $(*,*,1)$ or $(*,*,2,1)$. Let $a$ be the unique generator with the lowest Maslov grading.   Assume towards a contradiction that $a$ has no predecessor. Then $a$ has a successor $b$, and $b$ is  necessarily a turning point; let $c$ be the successor of $b.$ {In each of the following cases we will show that the ray in the positive direction of $\wb_s$ must intersect the bigon between $a$ and $b$, contradicting the fact that $\wb_s$ is an embedded curve.}
\begin{itemize}
    \item Case of $(*,*,1)$. There are no turning points after $b$, so all the remaining generators lie between $a$ and $c$ on $\wa$ (and in particular, between $a$ and $b$). {Therefore the ray in the positive direction of $\wb_s$ must intersect the bigon between $a$ and $b$, a contradiction.}
    \item Case of $(*,*,2,1).$ {By considering the number of generators in each Maslov grading, we can see that there are two} possibilities: 
    \begin{enumerate}[label=\arabic*.)]
       \item  \label{it:presuc_subcase3}$c$ is not a turning point, \change{and }neither is its \change{successor}, $d$. The successor of $d$, denoted by $e$, is a turning point; 
          \item $c$ is a turning point. \label{it:presuc_subcase1}
    \end{enumerate}
      {We often find it helpful to record diagrammatically the Maslov grading and the direction of $\wb_s$. See Figure \ref{fig: maslovdiagram}\subref{subfig: maslovdiagram_1} 
 and \ref{fig: maslovdiagram}\subref{subfig: maslovdiagram_2}  for the corresponding diagram of case~\ref{it:presuc_subcase3} and   case~\ref{it:presuc_subcase1}
    respectively.}\\
    Case \ref{it:presuc_subcase3}: note that $b$ and $d$ are the only generators in the second lowest Maslov grading. It follows that there are no more turning points after $e.$ But again in this case the remaining generators  lie between $a$ and $b$ on $\wa,$ {therefore the ray in the positive direction of $\wb_s$ must intersect the bigon between $a$ and $b$, a contradiction. }\\
     Case \ref{it:presuc_subcase1}:  There has to be a turning point after $c$, {since otherwise the ray in the positive direction of $\wb_s$ must intersect the bigon between $a$ and $b$.} Moreover this turning point has to be a negative turning point as else we would have the the incorrect number of generators of the Maslov grading of the successor. Let $f$ be the successor of this turning point, {and  $g$ be the successor of $f$. Since $b$ and $f$ are the only generators in the second lowest Maslov grading,  $f$ must be a non-trivial term in $\partial g$, and either $g$ is the final generator, or $g$ is a positive turning point and there are no more turning points after. In either case, we observe that the ray in the positive direction of $\wb_s$ must intersect the bigon between $a$ and $b$, a contradiction.  See Figure \ref{fig: maslovdiagram}\subref{subfig: maslovdiagram_2} for a schematic picture of the Maslov gradings of the intersection points in this case.} 
\end{itemize}
\end{proof} 
\begin{figure}
\begin{minipage}{.5\linewidth}
     \centering
      \subfloat[{Case \ref{it:presuc_subcase3} in the proof of Lemma \ref{le:presec}.}]{
       \begin{tikzpicture}[scale=0.9]
       \begin{scope}[thin, black!0!white]
	  \draw  (-3, 0) -- (6,0);
      \end{scope}

      	\filldraw (0, 0) circle (2pt) node[] (a){};
  	\filldraw (0.3, 1) circle (2pt) node[] (b){};
     	\filldraw (0.6, 2) circle (2pt) node[] (c){};
      	\filldraw (0.9, 1) circle (2pt) node[] (d){};
            	\filldraw (1.2, 2) circle (2pt) node[] (e){};

            \filldraw (1.5, 3) circle (2pt) node[] (f){};
            \filldraw (1.8, 2) circle (2pt) node[] (g){};
                 \filldraw (2.1, 3) circle (2pt) node[] (h){};
             
     \draw [ ->]  (a) -- (b)  ;
     \draw [ ->]   (b) -- (c) ;
      \draw [ ->]   (c)  -- (d) ;
       \draw [ ->]   (d)  -- (e) ;
        \draw [ ->]  (e)  -- (f) ;
         \draw [ ->]   (f)  -- (g) ;
        \draw [ ->]   (g)  -- (h);

      	 	\node  [left] at (a) {$a$};
         \node  [left] at (b) {$b$};
         \node  [left] at (c) {$c$};
         \node  [right] at (d) {$d$};
                \node  [left] at (e) {$e$};
         
    \end{tikzpicture}  \label{subfig: maslovdiagram_1}
     }
\end{minipage}%
\begin{minipage}{.5\linewidth}
     \centering
      \subfloat[{Case \ref{it:presuc_subcase1} in the proof of Lemma \ref{le:presec}.}]{
       \begin{tikzpicture}[scale=0.9]
       \begin{scope}[thin, black!0!white]
	  \draw  (-3, 0) -- (6,0);
      \end{scope}

      	\filldraw (0, 0) circle (2pt) node[] (a){};
  	\filldraw (0.3, 1) circle (2pt) node[] (b){};
     	\filldraw (0.6, 2) circle (2pt) node[] (c){};
      	\filldraw (0.9, 3) circle (2pt) node[] (d){};
            	\filldraw (1.2, 2) circle (2pt) node[] (e){};

            \filldraw (1.5, 3) circle (2pt) node[] (f){};
            \filldraw (1.8, 2) circle (2pt) node[] (g){};
                 \filldraw (2.1, 1) circle (2pt) node[] (h){};
                 
                    \filldraw (2.4, 2) circle (2pt) node[] (i){};
                         \filldraw (2.7, 3) circle (2pt) node[] (j){};
                                                  \filldraw (3, 2) circle (2pt) node[] (k){};
             
     \draw [ ->]  (a) -- (b)  ;
     \draw [ ->]   (b) -- (c) ;
      \draw [ ->]   (c)  -- (d) ;
       \draw [ ->]   (d)  -- (e) ;
        \draw [ ->]  (e)  -- (f) ;
         \draw [ ->]   (f)  -- (g) ;
        \draw [ ->]   (g)  -- (h);
          \draw [ ->]   (h)  -- (i);
            \draw [ ->]   (i)  -- (j);
              \draw [ ->]   (j)  -- (k);

      	 	\node  [left] at (a) {$a$};
         \node  [left] at (b) {$b$};
         \node  [left] at (c) {$c$};

          \node  [right] at (h) {$f$};
         
    \end{tikzpicture} \label{subfig: maslovdiagram_2}
     }
\end{minipage}
    \caption{The height indicates the Maslov grading, and the arrow indicates the direction of $\wb_s.$ Note that a point is a turning point if and only if it is between two consecutive arrows of the same direction, and it is a positive (resp.~negative) turning point if both arrows point upwards (resp.~downwards). }
    \label{fig: maslovdiagram}
\end{figure}
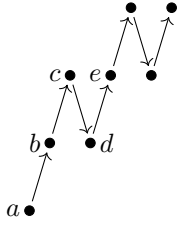
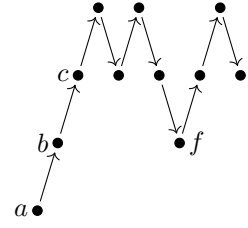

\begin{proposition} \label{prop:neccesary}
    If a standard diagram represents an almost L-space knot, then it is virtually almost coherent.
\end{proposition}
\begin{proof}
 Let $s$ be the exceptional spin$^c$ structure. Up to reflecting $\R^2$, Lemma~\ref{le:threearrows} implies that the number of generators in $\mathfrak{T}(\hH,s)$ in each Maslov grading is either $(*,*,1)$ or $(*,*,2,1)$ for some pair of integers $m,n$. {In each case we will prove that there are exactly two inconsistent arcs in $\wb_s$. Since the sign of the staircases only depend on whether the almost L-space knot is positive or negative, and therefore are uniform across all spin$^c$ structures, this will prove the proposition.}\\
 In both cases, let $a$ be the unique generator in the lowest Maslov grading. According to Lemma \ref{le:presec}, $a$ has a predecessor $b$ and a successor $c$. {Recall that a turning point, its predecessor and successor together occupy three consecutive Maslov gradings. In the case $(*,*,1)$, since $a$ is the unique generator in its Maslov grading, there can be no turning point in $\wa \cap \wb_s \setminus \{a,b,c\}$. 
 Denote by $\gamma_{pp'}$ the arc from intersection point $p$ to $p'$. It follows that
 all the arcs outside of $\gamma_{ba}$ and $\gamma_{ac}$  are in the same direction. Moreover, they must all be in the opposite direction as $\gamma_{ba}$ and $\gamma_{ac}$, because otherwise $\wb_s$ would be graphic, contradicting the fact that it represents an almost staircase. We conclude that there are exactly two inconsistent arcs in $\wb_s$ as claimed.
 \\
 In the case $(*,*,2,1)$, $b$ and $c$ are the only generators in their Maslov grading, therefore they admit both an incoming and an outgoing arrow to adjacent generators, and so they are both turning points. Let $d$ be the predecessor of $b$ and $e$ be the successor of $c$.  Since $b$ and $c$ are the only generators in their Maslov grading, there is no turning point in $\wa \cap \wb_s \setminus \{a,b,c,d, e\}$.  
 It follows that all the other arcs must be all in the same direction as $\gamma_{ba}$ and $\gamma_{ac}$, or as $\gamma_{ce}$ and $\gamma_{db}$.  
 Thus there are exactly two inconsistent arcs in $\wb_s$ as claimed.
 } 
\end{proof}

We now prove the converse to Proposition~\ref{prop:neccesary}. As claimed earlier, we will prove a stronger version of the converse: that all $(1,1)$ almost L-space knots have the $\CFKi$-type of complexes consisting of a staircase and a box of length $1.$ 
Before the proof, we rule out a corner case. The following lemma rules out the possibility that for a virtually almost coherent $(1,1)$ diagram, $\CFKi(Y,K,s)$ contains a single staircase (with varying signs) in each $s\in \Spin^c(Y).$ It turns out that Definition \ref{def: vir_almost} implies that in the exceptional spin$^c$ structure,  $\wb_s$ contains at least two arcs in each half plane,  one of them being the inconsistent arc.
\begin{figure}
\begin{minipage}{.5\linewidth}
     \centering
      \subfloat[]{
       \begin{tikzpicture}[scale=0.9]
       \begin{scope}[thin, red!50!white]
	  \draw  (-1, 0) -- (3,0);
   \draw  (-1, 2) -- (3,2);
   \draw  (-1, 3) -- (3,3);
      \end{scope}
    
   \foreach \i in {0,...,2}
   {
      	\filldraw (\i, 0) circle (2pt) node[] {};
       }

          \foreach \i in {0.5,...,2.5}
   {
      	\filldraw (\i, 3) circle (2pt) node[] {};
       }
       \draw [blue](0,0) -- (-0.2,-0.5);
         \draw[new arc arrow=to pos 0.7 with length 2mm, blue] (0,0) arc (180:0:0.5);
          \draw[new arc arrow=to pos 0.7 with length 2mm, blue] (1,0) arc (-180:0:0.5);
            \draw [blue](2,0) -- (2.5,3);
                  \draw[new arc arrow=to pos 0.7 with length 2mm, blue] (2.5,3) arc (0:180:0.5);
                        \draw[new arc arrow=to pos 0.7 with length 2mm, blue] (1.5,3) arc (0:-180:0.5);

        		\node  [above] at (2.2,0) {$3$};        	
         	\node  [above] at (1.2,0) {$2$};
          \node  [above] at (-0.2,0) {$1$};
           \node  [above] at (2.7,3) {$4$};
            \node  [above] at (1.3,3) {$5$};
             \node  [above] at (0.3,3) {$6$};
       \node  [above] at (3,0) {$\wa$};
              \node  [left] at (2.3,1.4) {$\wb_s$};
            \node  [above] at (-0.51,3) {$\wa'$};
    \end{tikzpicture}\label{subfig: cornercase_1}
     }
\end{minipage}%
\begin{minipage}{.5\linewidth}
     \centering
      \subfloat[]{
       \begin{tikzpicture}[scale=0.9]
       \begin{scope}[thin, red!50!white]
	  \draw  (-2, 0) -- (4,0);
   \draw  (-2, 3) -- (4,3);
       \end{scope}
   \begin{scope}[thin, black!50!white]
    \draw  (-2, 3) -- (-2,0);
     \draw  (4, 3) -- (4,0);    
      \end{scope}

      	\filldraw (-1, 0) circle (2pt) node[] {};
       \filldraw (1, 0) circle (2pt) node[] {};
       \filldraw (2.5, 0) circle (2pt) node[] {};

     \filldraw [black!60!white] (-0.8, 3) circle (2pt) node[] {};
      \filldraw [black!60!white] (3.3, 3) circle (2pt) node[] {};
       \filldraw (0.3, 3) circle (2pt) node[] {};
       \filldraw (2.3, 3) circle (2pt) node[] {};
       
         \draw[new arc arrow=to pos 0.7 with length 2mm, blue] (-1,0) arc (180:0:1);
          \draw[new arc arrow=to pos 0.7 with length 2mm, blue] (0.3, 3) arc (-180:0:1);

 \node  [above] at (-1, 3) {$1$};
 \node  [above] at (3.5, 3) {$1$};
\node  [above] at (0.5, 3) {$2$};
        		\node  [above] at (2.5,3) {$3$};    
\node  [above] at (-0.8, 0) {$1$};
\node  [above] at (1.2, 0) {$2$};
        		\node  [above] at (2.7,0) {$3$};        	
    
    \end{tikzpicture}\label{subfig: cornercase_2}
     }
\end{minipage}
\caption{The {curve in the} universal cover depicted in Figure \protect\subref{subfig: cornercase_1} cannot be a lift of {a $\beta$ curve in a} reduced $(1,1)$ diagram. Figure \protect\subref{subfig: cornercase_2} is a schematic picture of the standard $(1,1)$ diagram that corresponds to Figure \protect\subref{subfig: cornercase_1}, with the possible locations of the intersection point $1$ depicted. The intersection points $4$ and $5$ are next to $1$ and $2$ on the bottom side.}
\label{fig: cornercase}
\end{figure}
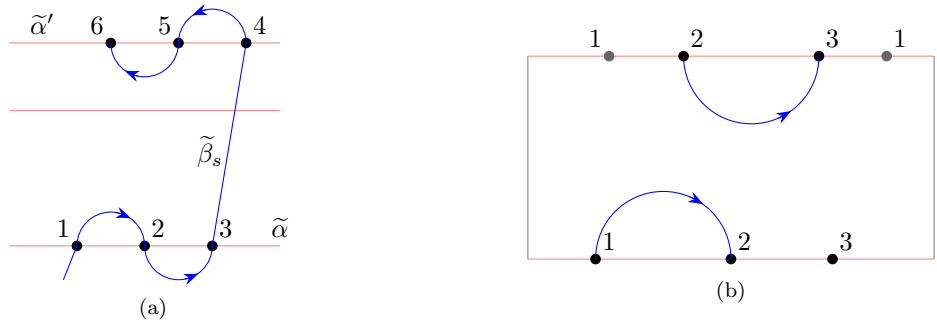
\begin{lemma}\label{le: cornercase}
    { Suppose $(S^1\times S^1,\alpha,\beta)$ is a virtually almost coherent $(1,1)$ diagram. If $\wb_s$ is the lift of $\beta$ that }contains the inconsistent arcs then $\wb_s$ contains at least two arcs in each half plane.
\end{lemma}
\begin{proof}
    Assume towards the contradiction that $\wa$ cuts $\wb_s$ into one arc and one ray in each half plane.  Without loss of generality, we may assume $\wb_s$ starts from the lower half plane and the graphic curves goes from left to right. According to the definition, $\wb_s$ {is graphic with respect to some other lift of $\alpha$} (with the opposite direction). Let $\wa'$ be the first such lift after $\wa$. {Label the intersection points $\wb_s\cap(\wa\cup\wa')$ by positive integers according to the order in which \change{they} appear on 
    \change{$\wb_s$} as an oriented curve}.    See Figure \ref{fig: cornercase} \subref{subfig: cornercase_1}.
    
    Project the intersection points to the standard $(1,1)$ diagram. Note that since the top and bottom sides of the standard $(1,1)$ diagram are identified by a shift, the cyclic order of the numbering (say) from left to right is invariant on each side. The cyclic order of the first three intersection points is $(1,2,3)$ on both top and bottom side. Now consider $4$ and $5$ on the bottom side. Since they are end points of a rainbow arc, the possible cyclic orders are $(1,5,4,2,3)$ or $(5,1,2,4,3)$. However,   $(1,5,4,2,3)$ cannot be a cyclic order on the top side, since $5$ is between $1$ and $4$, but $5$ is an end point of a rainbow arc {on the top side while $1$ and $4$ are not}. Similarly $(5,1,2,4,3)$ also cannot be a cyclic order on the top side, since $4$ is between $2$ and $3$, which are two end points of a rainbow arc while $4$ is not an end point of any rainbow arc. We have arrived at a contradiction.
\end{proof}
\begin{figure}[htb!]
     \centering
       \begin{tikzpicture}[scale=0.9]
       \begin{scope}[thin, red!50!white]
	  \draw  (-1, 0) -- (6,0);
      \end{scope}    
       \draw[arc arrow=to pos 0.7 with length 2mm, blue] (0,0) arc (180:0:1);     
           \draw[arc arrow=to pos 0.7 with length 2mm, blue] (2,0) arc (-180:0:1);   
           \node  [above] at (-0.3,0) {$c$};
          \node  [above] at (2.2,0) {$b$};
          \node  [above] at (4.3,0) {$a$};
              \node  [above] at (1,0) {$w$};
          \node  [below] at (3,0) {$z$};
    \end{tikzpicture}
\caption{}\label{fig: for_graphic_lemma}
\end{figure}
{We also need the following lemma:
\begin{lemma}\label{le:for_proof_sufficient}
Suppose a portion of $\wb_s$ is graphic and passes through intersection points $a,b$ and $c$ in order.
Further assume that before passing $a$, $\wb_s$ intersects $\wa$ only in $(-\infty,a) \subset \wa$ and after passing $c$, $\wb_s$ intersects $\wa$ only in $(c,\infty) \subset \wa$. Then viewed as a generator in $\CFKi(Y,K,s),$ $b$ admits exactly $2$ arrows, either both incoming arrows, from $a$ and $c$ respectively, or both outgoing arrows, to $a$ and $c$ respectively.
\end{lemma}}
{\begin{proof}
    Taking the reflection along $\wa$ if needed, the situation is depicted in Figure \ref{fig: for_graphic_lemma}. Suppose that there is a component of $\partial$ from $b$ to $p$ for some $p \in \wa \cap \wb_s.$ 
    Since there are no other intersection points in the intervals $(a,b)$ or $(b,c) $ on $\wa$, the boundary of the bigon from $b$ to $p$ must contain  the entirety of  $(a,b)$ or $(b,c)$ on $\wa$. Suppose the former case, then $p \in (-\infty,a]$ on $\wa$, and the boundary of the bigon from $b$ to $p$ must contain also the entire portion of $\wb_s$ from $b$ to $p$. Since $\wb_s$ intersects $\wa$ at $a$, the only such embedded bigon is the bigon from $b$ to $a$, and $p=a.$ Similarly, in the other case we have $p=b.$ Suppose there is a component of the differential from $p$ to $b$ for some  $p \in \wa \cap \wb_s.$ By the same argument above, the boundary of the bigon from $p$ to $b$ must contain the entirety of the boundary of either  the bigon between $a$ and $b$ or the bigon  between $b$ and $c$. There is no such embedded bigon. We conclude that $b$ admits exactly $2$ arrows, both of which outgoing, to $a$ and $c$ respectively.
    Since reflecting along $\wa$ reverses the direction of all arrows, the claim follows. 
\end{proof}}

\begin{proposition} \label{prop:sufficient}
    Virtually almost coherent diagrams represent almost L-space knots. Moreover, for such knots, the knot Floer complex in the exceptional spin$^c$ structure consists of a staircase direct summed with a box with length $1$. 
\end{proposition}
\begin{proof}
Let $\gamma$ and $\eta$ be the two inconsistent arcs in the lower and upper planes, respectively. There are two cases depending on whether  $\gamma$ and $\eta$ share a common end point or not.

First consider the case in which $\gamma$ and $\eta$ have no common end point. Let $\partial \gamma = \{b, d\}$ and $\partial \eta = \{c, e \}.$ Observe that one of the end points of $\gamma$ is between $c$ and $e$; let it be $b.$ Without loss of generality, we may assume that $d$ is to the right of $b$ and $c$ is to the right of $e.$ This is depicted in Figure \ref{fig: noshare} \subref{subfig: noshare_1}. Note that the remainder of the diagram is graphic to the right of $e$ and to the left of $d$. {Let $f$ \change{denote the} other intersection point adjacent to $d$. Then $f$}
must be outside the interval $[e,d]$ on $\wa$, because otherwise there will be at least $2$ basepoints in the bigon formed by $\gamma$ and $\wa.$ {Similarly $g$ must \change{be} located outside the interval $[e,d]$ on $\wa$, where $g$ is the other intersection point adjacent to $e$.}   Thus the differentials in the diagram satisfy
\begin{align*}
    \partial a &= b +c \\
    \partial b = \partial c &= e + d\\
    {\partial d = \partial e} & {= f + g}
\end{align*}
 {Since the rest of the curve is graphic, by Lemma \ref{le:for_proof_sufficient}, the generators $a,b,c,d$ and $e$ admit no other arrows than those depicted in Figure \ref{fig: noshare} \subref{subfig: noshare_2}. It follows that if we perform the filtered change of basis given by
 $\{d,e\}\mapsto \{d+e,e\}$, keeping everything else unchanged, this splits off a length one box summand, generated by $\{a,b,c,d+e\}.$ By Lemma \ref{le:graphic_implies_staircase} the rest of the generators $\{\cdots,e,f,g\}$ form a staircase.
 }
 Therefore $\CFKi(Y,K,s)$ is the direct sum of a staircase and a box with length $1$ as required.
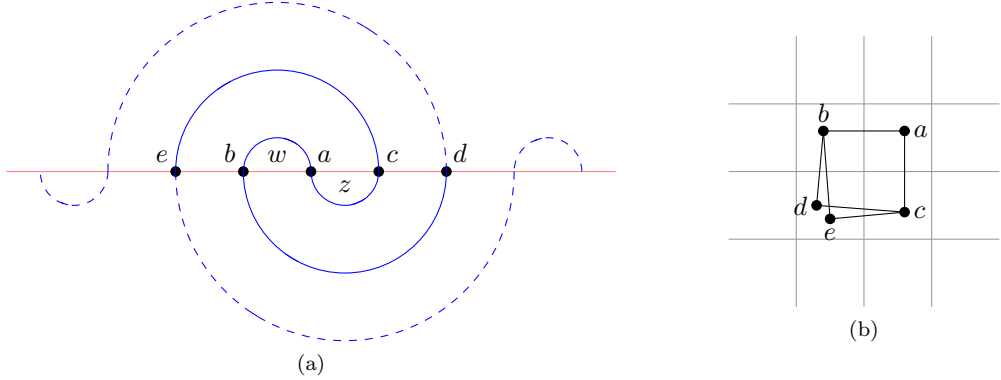
\begin{figure}
\begin{minipage}{.5\linewidth}
     \centering
      \subfloat[]{
       \begin{tikzpicture}[scale=0.9]
       \begin{scope}[thin, red!50!white]
	  \draw  (-2.5, 0) -- (6.5,0);
      \end{scope}
    
   \foreach \i in {-1,...,5}
   {
      	\filldraw (\i, 0) circle (2pt) node[] {};
       }
         \draw[arc arrow=to pos 0.7 with length 2mm,dashed, blue] (-1,0) arc (0:-180:0.5);
               \draw[arc arrow=to pos 0.7 with length 2mm,dashed, blue] (6,0) arc (0:180:0.5);
        \draw[arc arrow=to pos 0.7 with length 2mm,dashed, blue] (-1,0) arc (180:0:2.5);
     \draw[arc arrow=to pos 0.7 with length 2mm,dashed, blue] (5,0) arc (0:-180:2.5);
     
       \draw[arc arrow=to pos 0.7 with length 2mm, blue] (1,0) arc (180:0:0.5);
         \draw[arc arrow=to pos 0.7 with length 2mm, blue] (2,0) arc (-180:0:0.5);
           \draw[arc arrow=to pos 0.7 with length 2mm, blue] (3,0) arc (0:180:1.5);
           \draw[arc arrow=to pos 0.7 with length 2mm, blue] (4,0) arc (0:-180:1.5);

        		\node  [above] at (2.2,0) {$a$};
         	\node  [above] at (3.2,0) {$c$};
         	\node  [above] at (0.8,0) {$b$};
          \node  [above] at (4.2,0) {$d$};
          \node  [above] at (-0.2,0) {$e$};
            \node  [above] at (4.8,0) {$g$};
          \node  [above] at (-1.2,0) {$f$};
              \node  [above] at (1.5,0) {$w$};
          \node  [below] at (2.5,0) {$z$};
         
    \end{tikzpicture}\label{subfig: noshare_1}
     }
\end{minipage}%
\begin{minipage}{.5\linewidth}
     \centering
      \subfloat[]{
       \begin{tikzpicture}[scale=0.9]
        \begin{scope}[thin, black!0!white]
      \end{scope}
      \begin{scope}[thin, black!40!white]  
      \foreach \i in {-1,...,1}
   {
   \draw  (-2, \i) -- (2, \i);
     \draw  (\i, -2) -- (\i, 2);
       }
      \end{scope}
      
      	\filldraw (0.6, 0.6) circle (2pt) node[] (a){};
       \filldraw (0.6, -0.6) circle (2pt) node[] (c){};
       \filldraw (-0.6, 0.6) circle (2pt) node[] (b){};
              \filldraw (-0.5, -0.7) circle (2pt) node[] (e){};
              \filldraw (-0.7, -0.5) circle (2pt) node[] (d){};
              \filldraw (-1.5, -0.7) circle (2pt) node[] (f){};
              \filldraw (-0.7, -1.5) circle (2pt) node[] (g){};

     \draw  (0.6, 0.6) -- (0.6, -0.6);
          \draw  (0.6, 0.6) -- (-0.6, 0.6);
        \draw  (0.6, -0.6) -- (-0.5, -0.7);
             \draw  (0.6, -0.6) -- (-0.7, -0.5);
                  \draw  (-0.6, 0.6) -- (-0.5, -0.7);
                       \draw  (-0.6, 0.6) -- (-0.7, -0.5);
      \draw  (-0.5, -0.7) -- (-1.5, -0.7);
                       \draw  (-0.5, -0.7) -- (-0.7, -1.5);
                         \draw  (-0.7, -0.5) -- (-1.5, -0.7);
                       \draw  (-0.7, -0.5) -- (-0.7, -1.5);
  
      	 	\node  [right] at (a) {$a$};
         	\node  [right] at (c) {$c$};
          	\node  [above] at (b) {$b$};
           	\node  [above left] at (d) {$d$};
            \node  [below right] at (e) {$e$};
         	\node  [left] at (f) {$f$};
            \node  [below] at (g) {$g$};
         
    \end{tikzpicture} \label{subfig: noshare_2}
     }
\end{minipage}
    \caption{The case when the two distinct arcs do not share a common end point.}
    \label{fig: noshare}
\end{figure}

Now suppose $\gamma$ and $\eta$ share a common end point. Let $\partial \gamma = \{a, c\}$ and $\partial \eta = \{a, b \}.$ Without loss of generality, we may assume that $c$ is to the right of $b$. We may also assume without loss of generality that $\wb_s$ is oriented such that it passes through $c,a$ and $b$ in that order; the same argument holds for the other orientation. Let $d$ be the successor of $b$, and $e$ the predecessor of $c$. There are three subcases depending on whether $d$ and $e$ are in the interval $[b,c]$ on $\wa.$
\begin{itemize}
    \item Suppose both $d$ and $e$ are in the interval $[b,c]$. Note that in this case $d$ has a successor $f$ and $e$ has a predecessor $g$, as depicted in Figure \ref{fig: share} \subref{subfig: share_1}. The remaining diagram is graphic to the right of $f$ and to the left of $g$. Moreover, suppose $f$ has a successor {$h$, or $g$ has a predecessor $i$},  then they are necessarily outside the interval $[b,c].$ In particular, this implies the differentials in the diagram satisfy
    \begin{align*}
            \partial a &= b +c + g + f \\
     \change{ \partial b}&\change{= d+i\quad \partial c = e +  {h}}\\
   \change{ \partial g}&\change{= e+i\quad \partial f =  d + {h}.}
    \end{align*}
    {By Lemma \ref{le:for_proof_sufficient}, the generators $a$ through $g$ admit no other arrows than those depicted in \ref{fig: share} \subref{subfig: share_2}. Perform the filtered change of basis $\{b,c,d,e,g,f\} \mapsto 
    \{b+g,c+f,d+e,b,d,f\},$  keeping everything else unchanged. This splits off a box summand with side length $1$ generated by $\{ a,b+g,c+f, d+e \}$.  By Lemma \ref{le:graphic_implies_staircase} the remaining generators $\{\cdots,i,b,d,f,h\}$ form a staircase summand. }
    \item Suppose both $d$ and $e$ are outside the interval $[b,c]$. This is depicted in Figure \ref{fig: share} \subref{subfig: share_3}, {where the first intersection point to the right of $d$ (resp.~to the left of $e$) is denoted  $f$ (resp.~$g$)}.  Note that the remaining diagram is graphic to the right of $d$ and to the left of $e$. The differentials in the diagram are 
\begin{align*}
    &\partial a = b +c \\
    \partial b = &\partial c = e + d\\
    {\partial f =} &  {d  \quad 
     \partial g = e.}
\end{align*}
 {By Lemma \ref{le:for_proof_sufficient}, the generators $a,b,c,d $ and $e$ admit no other arrows than those depicted in Figure \ref{fig: share} \subref{subfig: share_4}. Perform the filtered change of basis $\{b,f,d,e\} \mapsto 
    \{b,b+f,d+e,e\},$ keeping everything else unchanged. This splits off a box summand with side length $1$ generated by $\{ a,b,c, d+e \}$.  By Lemma \ref{le:graphic_implies_staircase} the remaining generators $\{\cdots,b+f,e,g\}$ form a staircase summand. }
 \item Finally suppose $d$ is in the interval $[b,c]$ while $e$ is outside. The other case will follow  in parallel. As depicted in Figure \ref{fig: share} \subref{subfig: share_5}, in this case $d$ has a successor $f$.
  The differentials in the diagram are
  \begin{align*}
      \partial a &= b + f +c \\
      \partial b &= d + e \quad
      \partial c =  {e + h}\\
      \partial f &=  {d + h} \quad
       {\partial g = e.}
  \end{align*}
  {By Lemma \ref{le:for_proof_sufficient}, the generators $a$ through $f$ admit no other arrows than those depicted in Figure \ref{fig: share} \subref{subfig: share_6}. Perform the filtered change of basis $\{c,d,e,f\} \mapsto 
    \{c+f,d+e,c,e\}$, keeping everything else unchanged. This splits off a box summand with length $1$ generated by $\{ a,b,c+f, d+e \}$.  By Lemma \ref{le:graphic_implies_staircase} the remaining generators $\{\cdots,g,e,c,h\}$ form a staircase summand. }
\end{itemize}

\begin{figure}
\begin{minipage}{.5\linewidth}
     \centering
      \subfloat[]{
       \begin{tikzpicture}[scale=0.9]
       \begin{scope}[thin, red!50!white]
	  \draw  (-2, 0) -- (8,0);
      \end{scope}
    
   \foreach \i in {-1,...,7}
   {
      	\filldraw (\i, 0) circle (2pt) node[] {};
       }
       
     \draw[arc arrow=to pos 0.7 with length 2mm, dashed, blue] (2,0) arc (-180:0:2.5);
   \draw[arc arrow=to pos 0.7 with length 2mm, dashed, blue] (-1,0) arc (180:0:2.5);
       \draw[arc arrow=to pos 0.7 with length 2mm, blue] (0,0) arc (180:0:1.5);
         \draw[arc arrow=to pos 0.7 with length 2mm, blue] (0,0) arc (-180:0:0.5);
         \draw[arc arrow=to pos 0.7 with length 2mm, blue] (1,0) arc (180:0:0.5);
           \draw[arc arrow=to pos 0.7 with length 2mm, blue] (3,0) arc (-180:0:1.5);
           \draw[arc arrow=to pos 0.7 with length 2mm, blue] (4,0) arc (-180:0:0.5);
             \draw[arc arrow=to pos 0.7 with length 2mm, blue] (5,0) arc (180:0:0.5);

        		\node  [above] at (3.2,0) {$a$};
         	\node  [above] at (6.2,0) {$c$};
          	\node  [above] at (4.8,0) {$e$};
         	\node  [above] at (0.8,0) {$d$};
                 \node  [above] at (2.2,0) {$f$};
          \node  [above] at (3.8,-0.5) {$g$};
          \node  [above] at (-0.2,0) {$b$};
            \node  [above] at (-0.8,0) {$i$};
            \node  [above] at (6.9,0) {$h$};
              \node  [above] at (1.5,0) {$w$};
          \node  [below] at (4.5,0) {$z$};
          \node  [above] at (5.5,0) {$w$};
           \node  [below] at (0.5,0) {$z$};
         
    \end{tikzpicture}\label{subfig: share_1}
     }
\end{minipage}%
\begin{minipage}{.5\linewidth}
     \centering
      \subfloat[]{
       \begin{tikzpicture}[scale=0.9]
    
      \begin{scope}[thin, black!40!white]  
      \foreach \i in {-1,...,1}
   {
   \draw  (-2, \i) -- (2, \i);
     \draw  (\i, -2) -- (\i, 2);
       }
      \end{scope}
      
      	\filldraw (0.6, 0.6) circle (2pt) node[] (a){};
       \filldraw (0.4, -0.4) circle (2pt) node[] (f){};
              \filldraw (0.6, -0.6) circle (2pt) node[] (c){};
               \filldraw (0.4, -1.4) circle (2pt) node[] (h){};
                     \filldraw (-0.4, 0.4) circle (2pt) node[] (g){};
                          \filldraw (-1.4, 0.4) circle (2pt) node[] (i){};
       \filldraw (-0.6, 0.6) circle (2pt) node[] (b){};
              \filldraw (-0.5, -0.7) circle (2pt) node[] (e){};
              \filldraw (-0.7, -0.5) circle (2pt) node[] (d){};

     \draw  (0.6, 0.6) -- (0.6, -0.6);
      \draw  (0.6, 0.6) -- (0.4, -0.4);
       \draw  (0.6, 0.6) -- (-0.4, 0.4);
          \draw  (0.6, 0.6) -- (-0.6, 0.6);
        \draw  (0.6, -0.6) -- (-0.5, -0.7);
             \draw  (0.4, -0.4) -- (-0.7, -0.5);
                  \draw  (-0.4, 0.4) -- (-0.5, -0.7);
                       \draw  (-0.6, 0.6) -- (-0.7, -0.5);
                        \draw  (0.6, -0.6) -- (0.4, -1.4);
             \draw  (0.4, -0.4) -- (0.4, -1.4);
                  \draw  (-0.4, 0.4) -- (-1.4, 0.4);
                       \draw  (-0.6, 0.6) -- (-1.4, 0.4);

      	 	\node  [right] at (a) {$a$};
         	\node  [right] at (c) {$c$};
          	\node  [above] at (b) {$b$};
           	\node  [left] at (d) {$d$};
            \node  [below] at (e) {$e$};
               \node  [below] at (0.25,0.15) {$f$};
                  \node  [below] at (-0.25,0.45) {$g$};
        \node  [above] at (i) {$i$};
      \node  [right] at (h) {$h$};
        
    \end{tikzpicture} \label{subfig: share_2}
     }
\end{minipage}\\
\begin{minipage}{.5\linewidth}
     \centering
      \subfloat[]{
       \begin{tikzpicture}[scale=0.9]
  \begin{scope}[thin, red!0!white]
	  \draw  (-2, 0) -- (8,0);
      \end{scope}    
       \begin{scope}[thin, black!50!white]
	  \draw  (-1, 0) -- (7,0);
      \end{scope}
    
   \foreach \i in {0,...,6}
   {
      	\filldraw (\i, 0) circle (2pt) node[] {};
       }
        \draw[arc arrow=to pos 0.7 with length 2mm,dashed, blue] (1,0) arc (0:-180:0.5);
     \draw[arc arrow=to pos 0.7 with length 2mm,dashed, blue] (6,0) arc (0:180:0.5);
       \draw[arc arrow=to pos 0.7 with length 2mm, blue] (2,0) arc (180:0:0.5);
         \draw[arc arrow=to pos 0.7 with length 2mm, blue] (3,0) arc (-180:0:0.5);
           \draw[arc arrow=to pos 0.7 with length 2mm, blue] (4,0) arc (0:180:1.5);
           \draw[arc arrow=to pos 0.7 with length 2mm, blue] (5,0) arc (0:-180:1.5);

        		\node  [above] at (3.2,0) {$a$};
         	\node  [above] at (4.2,0) {$c$};
         	\node  [above] at (1.8,0) {$b$};
          \node  [below] at (5.2,0) {$d$};
          \node  [below] at (6,0) {$f$};
          \node  [above] at (0.8,0) {$e$};
           \node  [above] at (0,0) {$g$};
              \node  [above] at (2.5,0) {$w$};
          \node  [below] at (3.5,0) {$z$};
         
    \end{tikzpicture}\label{subfig: share_3}
     }
\end{minipage}%
\begin{minipage}{.5\linewidth}
     \centering
      \subfloat[]{
       \begin{tikzpicture}[scale=0.9]
        \begin{scope}[thin, black!0!white]
      \end{scope}
      \begin{scope}[thin, black!40!white]  
      \foreach \i in {-1,...,1}
   {
   \draw  (-2, \i) -- (2, \i);
     \draw  (\i, -2) -- (\i, 2);
       }
      \end{scope}
      
      	\filldraw (0.6, 0.6) circle (2pt) node[] (a){};
       \filldraw (0.6, -0.6) circle (2pt) node[] (c){};
             \filldraw (0.6, -0.8) circle (2pt) node[] (g){};
       \filldraw (-0.6, 0.6) circle (2pt) node[] (b){};
        \filldraw (-0.8, 0.6) circle (2pt) node[] (f){};
              \filldraw (-0.5, -0.7) circle (2pt) node[] (e){};
              \filldraw (-0.7, -0.5) circle (2pt) node[] (d){};

     \draw  (0.6, 0.6) -- (0.6, -0.6);
          \draw  (0.6, 0.6) -- (-0.6, 0.6);
        \draw  (0.6, -0.6) -- (-0.5, -0.7);
         \draw  (0.6, -0.8) -- (-0.5, -0.7);
             \draw  (0.6, -0.6) -- (-0.7, -0.5);
                  \draw  (-0.6, 0.6) -- (-0.5, -0.7);
                    \draw  (-0.8, 0.6) -- (-0.7, -0.5);
                       \draw  (-0.6, 0.6) -- (-0.7, -0.5);

      	 	\node  [right] at (a) {$a$};
         	\node  [right] at (c) {$c$};
          	\node  [above] at (b) {$b$};
           	\node  [left] at (d) {$d$};
            \node  [left] at (f) {$f$};
            \node  [below] at (e) {$e$};
                    \node  [below] at (g) {$g$};

    \end{tikzpicture} \label{subfig: share_4}
     }
     \end{minipage}\\
\begin{minipage}{.5\linewidth}
     \centering
      \subfloat[]{
       \begin{tikzpicture}[scale=0.9]
       \begin{scope}[thin, red!50!white]
	  \draw  (-2, 0) -- (8,0);
      \end{scope}
    
   \foreach \i in {-1,...,6}
   {
      	\filldraw (\i, 0) circle (2pt) node[] {};
       }
       
            \draw[arc arrow=to pos 0.7 with length 2mm, dashed, blue] (0,0) arc (0:-180:0.5);
     \draw[arc arrow=to pos 0.7 with length 2mm, dashed, blue] (3,0) arc (-180:0:1.5);
   \draw[arc arrow=to pos 0.7 with length 2mm, blue] (0,0) arc (180:0:2.5);
       \draw[arc arrow=to pos 0.7 with length 2mm, blue] (1,0) arc (180:0:1.5);
         \draw[arc arrow=to pos 0.7 with length 2mm, blue] (1,0) arc (-180:0:0.5);
         \draw[arc arrow=to pos 0.7 with length 2mm, blue] (2,0) arc (180:0:0.5);
           \draw[arc arrow=to pos 0.7 with length 2mm, blue] (4,0) arc (-180:0:0.5);
           \draw[arc arrow=to pos 0.7 with length 2mm, dashed, blue] (6,0) arc (180:0:0.5);

        		\node  [above] at (3.2,0) {$f$};
          	\node  [above] at (5.2,0) {$c$};
         	\node  [above] at (0.8,0) {$b$};
                 \node  [above] at (1.8,0) {$d$};
          \node  [above] at (4.2,0) {$a$};
          \node  [above] at (-0.2,0) {$e$};
           \node  [above] at (-1,0) {$g$};
             \node  [above] at (5.8,0) {$h$};
              \node  [above] at (2.5,0) {$w$};
          \node  [below] at (4.5,0) {$z$};
        \node  [below] at (1.5,0) {$z$};

    \end{tikzpicture}\label{subfig: share_5}
     }
\end{minipage}%
\begin{minipage}{.5\linewidth}
     \centering
      \subfloat[]{
       \begin{tikzpicture}[scale=0.9]
        \begin{scope}[thin, black!0!white]
      \end{scope}
      \begin{scope}[thin, black!40!white]  
      \foreach \i in {-1,...,1}
   {
   \draw  (-2, \i) -- (2, \i);
     \draw  (\i, -2) -- (\i, 2);
       }
      \end{scope}
      
      	\filldraw (0.6, 0.6) circle (2pt) node[] (a){};
       \filldraw (0.4, -0.6) circle (2pt) node[] (c){};
              \filldraw (0.8, -0.4) circle (2pt) node[] (f){};
       \filldraw (-0.6, 0.6) circle (2pt) node[] (b){};
       \filldraw (-0.9, 0.4) circle (2pt) node[] (g){};
              \filldraw (-0.8, -0.6) circle (2pt) node[] (e){};
              \filldraw (-0.4, -0.4) circle (2pt) node[] (d){};
               \filldraw (0.6, -1.4) circle (2pt) node[] (h){};

     \draw  (0.6, 0.6) -- (-0.6, 0.6);
 \draw  (0.6, 0.6) -- (0.8, -0.4);
  \draw  (0.6, 0.6) -- (0.4, -0.6);
  
        \draw  (-0.6, 0.6) -- (-0.4, -0.4);
              \draw  (-0.6, 0.6) -- (-0.8, -0.6);
              \draw  (-0.9, 0.4) -- (-0.8, -0.6);
                \draw  (0.8, -0.4) -- (0.6, -1.4);
                   \draw  (0.4, -0.6) -- (0.6, -1.4);
             \draw  (0.8, -0.4) -- (-0.4, -0.4);
                       \draw  (0.4, -0.6) -- (-0.8, -0.6);

      	 	\node  at (0.8,0.82) {$a$};
         	\node  [below left] at (c) {$c$};
          	\node  [above] at (b) {$b$};
           	\node  [above right] at (d) {$d$};
            \node  [below] at (e) {$e$};
           \node  [right] at (f) {$f$};
           \node  [left] at (g) {$g$};
           \node  [right] at (h) {$h$};
         
    \end{tikzpicture} \label{subfig: share_6}
     }
     \end{minipage}
    \caption{The cases when the two distinct arcs share a common end point.}
    \label{fig: share}
\end{figure}
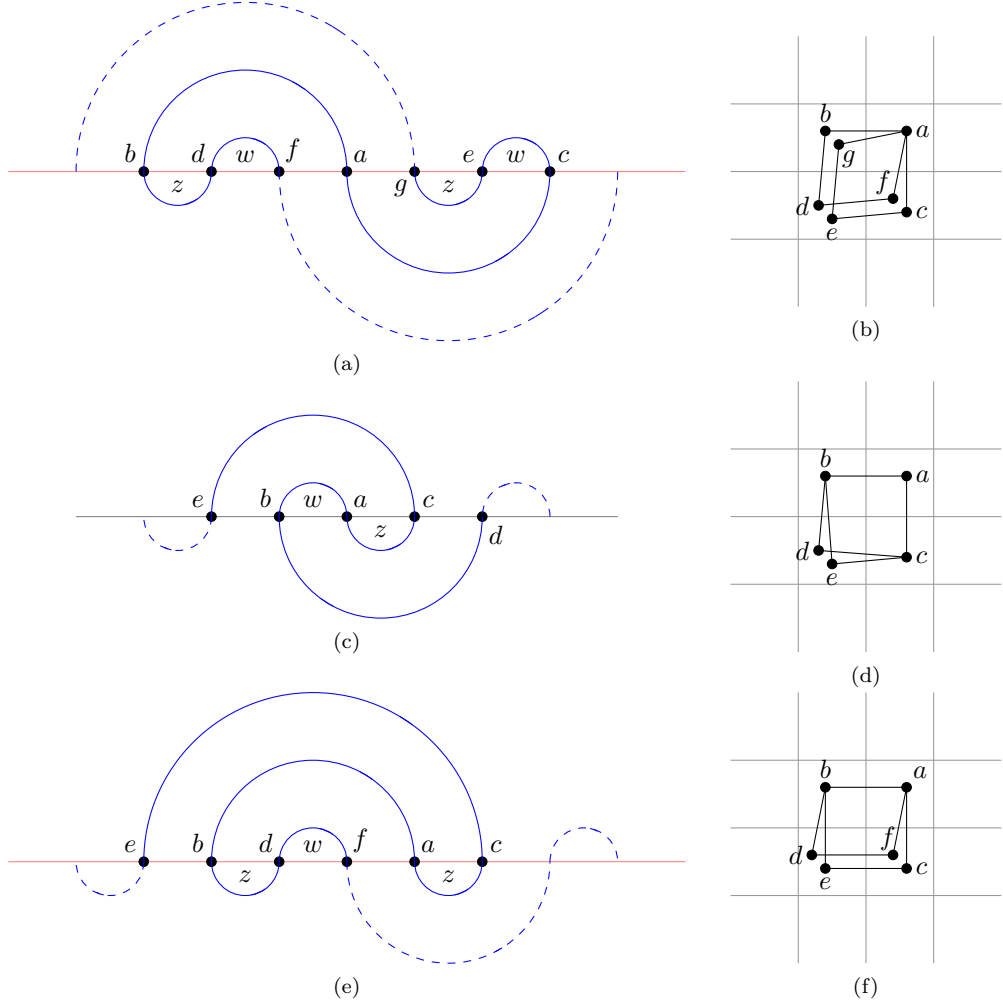
\end{proof}

\begin{figure}
\begin{minipage}{.5\linewidth}
\centering
\subfloat[The standard diagram of $K_j.$ The $i$-th intersection point on the top side is identified with the $(i+2)$-th point on the bottom side.]{
\labellist
 \pinlabel { $w$}  at 62 22
 \pinlabel { $z$}  at 132 160
  \pinlabel { $\cdots$}  at 60 130
   \pinlabel { $4j$}  at 60 145
\endlabellist
\includegraphics[scale=0.85]{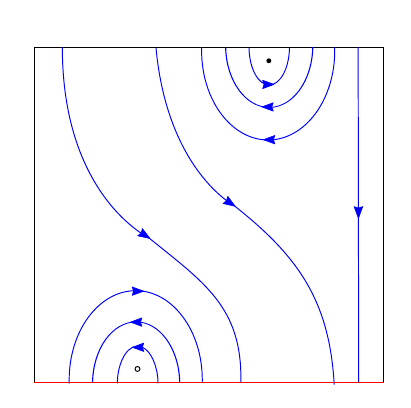}\label{subfig: kj_a}
}
\end{minipage}%
\begin{minipage}{.5\linewidth}
     \centering
      \subfloat[One iteration of the $\beta$ curve in the universal cover.]{
       \begin{tikzpicture}[scale=0.7]
       \begin{scope}[thin, red!50!white]
	  \draw  (-5.5, 0) -- (5.5,0);
      \end{scope}
      \begin{scope}[thin, black!50!white]
	  \draw  (-5.5, 2) -- (5.5,2);
   \draw  (-5.5, 3.5) -- (-2,3.5);
      \draw  (-5.5, -2) -- (5.5,-2);
         \draw  (-5.5, -4) -- (5.5,-4);
         \draw  (2, -5.5) -- (5.5,-5.5);
      \end{scope}
     \begin{scope}[thin, black!0!white]
          \draw  (0, -2) -- (0, 3);
      \end{scope}
      
        \draw[arc arrow=to pos 0.7 with length 2mm, blue] (0,0) arc (180:0:1.5);
        \draw[arc arrow=to pos 0.7 with length 2mm, blue] (0.5,0) arc (180:0:1);
        \draw[arc arrow=to pos 0.7 with length 2mm, blue] (-4.5,4) arc (180:0:0.25);
        \draw[arc arrow=to pos 0.7 with length 2mm, blue] (-3,-2) arc (-180:0:1.5);
        \draw[arc arrow=to pos 0.7 with length 2mm, blue] (-2.5,-2) arc (-180:0:1);
        \draw[arc arrow=to pos 0.7 with length 2mm, blue] (4,-6) arc (-180:0:0.25);
        \draw [blue] (-4,4) -- (-2.5,-2);
         \draw [blue] (-4.5,4) -- (-3,-2);
         \draw [blue] (-0.5,-2) -- (0.5,0);
         \draw [blue] (2.5,0) -- (4,-6);
         \draw [blue] (3,0) -- (4.5,-6);

  \node   at (5,0.4) {$0$};
  \node   at (5,-1.75) {$-1$};
  \node   at (5,-3) {\vdots};
  \node   at (5,-3.8) {\vdots};
  \node   at (5,-5.2) {$-j$};
 \node   at (-1.6,1.75) {$1$};
 \node   at (-1.6,3.55) {$j-1$};
 \node   at (-1.6,2.75) {\vdots};
         
    \end{tikzpicture}\label{subfig: kj_b}
     }
\end{minipage}\\
\begin{minipage}{\linewidth}
     \centering
      \subfloat[The lift to the universal cover, when $j\geq 1$.]{
       \begin{tikzpicture}[scale=0.8]
       \begin{scope}[thin, red!50!white]
	  \draw  (-7.4, 0) -- (7.4,0);
    \draw  (-8.4, 0) -- (-7.6,0);
   \draw  (8.4, 0) -- (7.6,0);
      \end{scope}
     \begin{scope}[thin, black!0!white]
          \draw  (0, -2) -- (0, 3);
      \end{scope}
      
   \foreach \i in {-8,...,8}
   {
      	\filldraw (\i, 0) circle (2pt) node[] {};
       }
        \draw[arc arrow=to pos 0.7 with length 2mm, blue] (-7,0) arc (180:0:0.5);
        \draw[arc arrow=to pos 0.7 with length 2mm, blue] (-6,-0.5) arc (-180:0:0.5);
        \draw[arc arrow=to pos 0.7 with length 2mm, blue] (-5,0) arc (180:0:0.5);
        \draw[arc arrow=to pos 0.7 with length 2mm, blue] (-4,0) arc (-180:0:2.5);
        \draw[arc arrow=to pos 0.7 with length 2mm, blue] (-3,0) arc (-180:0:1.5);
        \draw[arc arrow=to pos 0.7 with length 2mm, blue] (-2,0) arc (-180:0:0.5);
        \draw[arc arrow=to pos 0.7 with length 2mm, blue] (-3,0) arc (180:0:0.5);
        \draw[arc arrow=to pos 0.7 with length 2mm, blue] (-1,0) arc (180:0:2.5);
         \draw[arc arrow=to pos 0.7 with length 2mm, blue] (0,0) arc (180:0:1.5);
        \draw[arc arrow=to pos 0.7 with length 2mm, blue] (1,0) arc (180:0:0.5);
                \draw[arc arrow=to pos 0.7 with length 2mm, blue] (2,0) arc (-180:0:0.5);
        \draw[arc arrow=to pos 0.7 with length 2mm, blue] (4,0) arc (-180:0:0.5);
        \draw[arc arrow=to pos 0.7 with length 2mm, blue] (5,0.5) arc (180:0:0.5);
        \draw[arc arrow=to pos 0.7 with length 2mm, blue] (6,0) arc (-180:0:0.5);

   \node   at (-6,1) {$\overbrace{\ \hspace{4em} }^{j-1 \text{ copies}}$};
  \node   at (6,1.5) {$\overbrace{\ \hspace{4em} }^{j-1 \text{ copies}}$};

        \node   at (-8.2,0.3) {\small$x_{-3-2j}$};
\node   at (-7.5,-0.3) {$\cdots$};
 \node   at (-0.7,-0.3) {\small$x_{-1}$};
   \node   at (-2.3,-0.3) {\small$x_{-2}$};
    \node   at (-3.3,-0.3) {\small$x_{-3}$};
     \node   at (3.3,-0.3) {\small$x_{3}$};
 \node   at (0.2,-0.3) {\small$x_0$};
  \node   at (1,-0.3) {\small$x_1$};
   \node   at (1.8,-0.3) {\small$x_2$};
    \node   at (7.5,0.3) {$\cdots$};
        \node   at (8.45,-0.3) {\small$x_{3+2j}$};
     
           \node   at (-1.5,-2) {\small$z$};
          \node   at (-1.5,-0.2) {\small$z$};
                  \node   at (-6.5,0.2) {\small $w$};
                      \node   at (-4.5,0.2) {\small$w$};
          \node   at (-5.5,-0.2) {\small$z$};
           \node   at (-5.5,-0.6) {\small$z$};
         
           \node   at (1.5,2) {\small$w$};
          \node   at (1.5,0.2) {\small$w$};
                  \node   at (6.5,-0.2) {\small $z$};
                      \node   at (4.5,-0.2) {\small$z$};
          \node   at (5.5,0.2) {\small$w$};
           \node   at (5.5,0.6) {\small$w$};
              \node   at (-2.5,0.2) {\small$w$};
                 \node   at (2.5,-0.2) {\small$z$};
          \draw [blue] (-6,0) -- (-6,-0.5);
          \draw [blue] (-5,0) -- (-5,-0.5);
          \draw [blue] (6,0) -- (6,0.5);
          \draw [blue] (5,0) -- (5,0.5);
  \draw [blue] (-8,0) -- (-8,-2);
  \draw [blue] (8,0) -- (8,2);
        
    \end{tikzpicture}\label{subfig: kj_c}
     }
\end{minipage}%
\caption{Each knot $K_j = K(7+4j,3,4j,2)$ for $j\geq 0$ is strongly almost coherent.}
\label{fig: kj}
\end{figure}

We conclude this paper by proving our result on homology cobordism. {Recall that we have set $K_j$ to be the $(1,1)$ knot in $S^3$ given by the four-tuple $(7+4j,3,4j,2)$ with $j\in \Z_{\geq 0}$.}

\propkj*

The argument we use for claim \eqref{it:kj2} is well-known to experts in involutive knot Floer homology, and we write it down here for the sake of completeness.
\begin{proof}[Proof of Proposition \ref{prop: kj}]
    The lift to the universal cover is obtained by concatenating $2j$ iterations of the portion displayed in Figure \ref{fig: kj}\subref{subfig: kj_b} and then pulling tight the curve. Choose a basis $x_i$ with $|i| \leq 3+2j$ as shown in Figure  \ref{fig: kj}\subref{subfig: kj_c}, where $x_0$ is the middle intersection point. When $j \geq 1,$ the {non-trivial components of the differential are} given by
    \begin{align*}
        \partial x_{\pm 2} &= x_{\pm 3} + x_{\pm 1} \\
        \partial x_{\pm 3} &= \partial x_{\pm 1} = x_0 \\
        \partial x_{\pm 4} &= x_{\pm 5} + x_{\pm 3} + x_{\pm 1}\\
        \partial x_{\pm 2i} &= x_{\pm 2i+1} + x_{\pm 2i-1}  \qquad \text{for} \quad 3\leq i \leq j+1.
    \end{align*}
    After a filtered change of basis 
    \begin{align*}
        y_0 &= x_2 + x_{-2}\\
        y_{1} &= x_{3} + x_{-1} \\
        y_{-1} &= x_{-3} + x_{1} \\
        y_{\pm i} &= x_{\pm (i+2)} \qquad \text{for} \quad 2\leq i \leq 2j,
    \end{align*}
    we obtain a direct summand of a positive staircase $D_j$ generated by $\{y_i  | -2j-1 \leq i \leq 2j+1\}$ and a box generated by $\{x_0,x_1,x_2,x_3\}.$
    
    Quotienting out the box complex induces a local equivalence. The remaining staircase $D_j$  has $3+4j$ generators, with the top half of the staircase {(i.e. the subcomplex generated by $\{y_j:j\leq 0\}$)} \change{having total horizontal length} $n(D_j)=j+1$. Therefore we can apply \cite[Theorem 3.1]{hugo_concordance} (See also \cite[Remark 3.3]{hugo_concordance}) {to determine the filtered chain homotopy 
 type of $\CFKi(S^3_{+1}(K_j),\mu)$ where $\mu$ is the image of a meridian of  $K$ in the $+1$-surgery on $K_j$, $S^3_{+1}(K_j)$}. In particular, up to filtered chain homotopy, $\CFKi(S^3_{+1}(K_j),\mu)$ splits as the direct sum of an acylic piece and a summand \change{$C_{j+1}$}, as defined in ~\cite[Theorem 3.1]{hugo_concordance}.  {By \cite[Theorem 7.1]{OSknot}, the knot Floer complex of the connected sum is given by the tensor product.  Since $\CFKi(-S^3_{+1}(K_j),U)$ is generated by a single generator, tensoring with it amounts to a constant grading shift, which preserves the local equivalence class.  Consider the homomorphism $\bigoplus_{j>0}\varphi_{j+1,j}\colon \CZhat/\CZ \rightarrow \bigoplus_{j>0}\Z$  from~\cite[Theorem 1.1]{dai2021homology}. Since $\CFKi(S^3_{+1}(K_j)\conn -S^3_{+1}(K_j),\mu \conn U)$ is \change{locally} equivalent to a \change{$C_{j+1}$} summand, \change{which is a standard complex in the sense of~\cite[Definition 5.1]{dai2021homology}, it immediately follows from~\cite[Definition 8.1]{dai2021homology} that  $\varphi_{j+1,j}(S^3_{+1}(K_j),\mu)= -1$ for $j>0$  and $\varphi_{i+1,i}(S^3_{+1}(K_j),\mu)=0$ for $i\neq j$. This last step is essentially that carried out in the proof of~\cite[Proposition 11.2, Lemma 11.3]{dai2021homology}}.} This proves claim \eqref{it:kj1}.
    
    We prove claim \eqref{it:kj2} by considering the Hendricks-Manolescu's involution $\iota$ on $\CFKi(K_j).$ Since $\iota$ is a skew-filtered, grading-preserving chain map that squares to the Sarkar{'s \change{basepoint} pushing} map{~\cite{Sarkarmovingbasepoint}}, up to filtered homotopy {equivalence}  one may take $\iota$ to send $y_0 \mapsto y_0 + U^{-1}x_0$ and $x_2 \mapsto x_2 + y_0{,}$ and to be reflection about the line $i=j$ for every other element. (This is essentially the same computation as those given in the proof of~\cite[Proposition 8.1]{involutive}.)
    
Assume to the contradiction that $K_j$ is concordant to an L-space knot. Then as an $\iota_K$-complex (\cite[Definition 2.2]{connectedian}),
 $(\CFKi(K_j),\iota)$ is locally equivalent {(\cite[Definition 2.4]{connectedian})} to   $(D,\iota'),$  where $D$ is a staircase and $\iota'$ is the canonical involution. In fact, $D$ must be isomorphic to $D_j$, the staircase complex generated by all $y_i$. Let $f\colon  (\CFKi(K_j),\iota) \longrightarrow (D_j,\iota')$ be an $\iota_K$-local {equivalence}. As a graded, filtered map $f$ sends each $y_i$ to itself and either $x_2 \mapsto 0$ or $x_2 \mapsto y_0.$ In either case, {from the definition of local equivalence, we have that $f$ commutes up to homotopy with the involutions i.e.}
 \[
      H \partial (x_2) + \partial H(x_2) = f\iota(x_2) + \iota'f (x_2) = f(x_2 + y_0) + \iota'f(x_2) = y_0
 \]
 for some map $H$. It follows that $H(x_1 + x_3) + y_0 \in \operatorname{im}\partial.$ On the other hand, we have $f(x_1 + x_3) =  f\partial(x_2) = \partial f(x_2)= 0$ or $y_1 + y_{-1}.$  In either case, we compute
 \[
     0=   f \iota (x_1 + x_3) + \iota'f(x_1 + x_3) = H\partial (x_1 + x_3) + \partial H(x_1 + x_3) = y_1 + y_{-1},
 \]
    a contradiction. It follows that $K_j$ is not concordant to any L-space knot.
\end{proof}
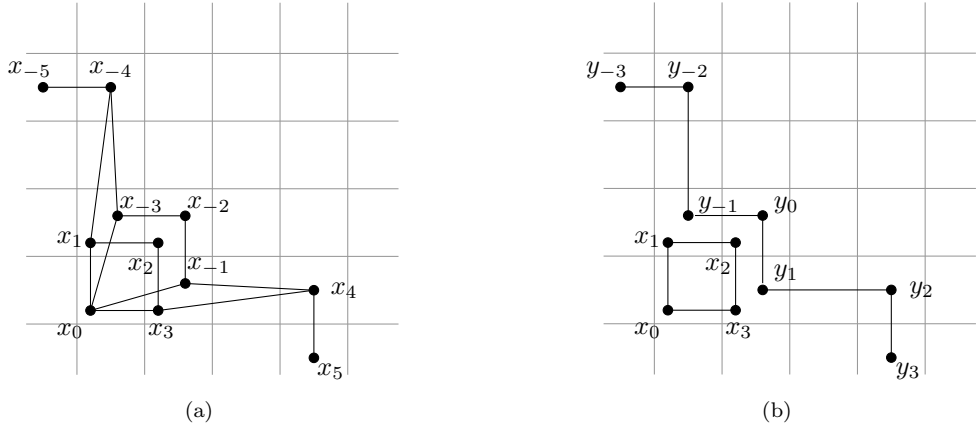
\begin{figure}
\begin{minipage}{.5\linewidth}
     \centering
   \subfloat[]{
    \centering
       \begin{tikzpicture}[scale=0.9]
      \begin{scope}[thin, black!40!white]  
      \foreach \i in {0,...,4}
   {
   \draw  (-0.75, \i) -- (4.75, \i);
     \draw  (\i, -0.75) -- (\i, 4.75);
       }
      \end{scope}
      \filldraw (0.2, 0.2) circle (2pt) node[] (){};
       \filldraw (0.2, 1.2) circle (2pt) node[] (){};
       \filldraw (1.2, 0.2) circle (2pt) node[] (){};
       \filldraw (1.2, 1.2) circle (2pt) node[] (){};

        \filldraw (0.6, 1.6) circle (2pt) node[] (){};
         \filldraw (1.6, 0.6) circle (2pt) node[] (){};
          \filldraw (1.6, 1.6) circle (2pt) node[] (){};
           \filldraw (0.5, 3.5) circle (2pt) node[] (){};
              \filldraw (3.5, 0.5) circle (2pt) node[] (){};
                 \filldraw (-0.5, 3.5) circle (2pt) node[] (){};
                    \filldraw (3.5, -0.5) circle (2pt) node[] (){};

     \draw  (3.5, 0.5) -- (3.5, -0.5);
        \draw  (0.5, 3.5) -- (-0.5, 3.5);
        
         \draw  (3.5, 0.5) -- (1.6, 0.6);
         \draw  (3.5, 0.5) -- (1.2, 0.2);
                \draw  (0.5, 3.5) -- (0.6, 1.6);
         \draw  (0.5, 3.5) -- (0.2, 1.2);
          \draw  (0.2, 0.2) -- (1.6, 0.6);
         \draw  (0.2, 0.2) -- (1.2, 0.2);
                \draw  (0.2, 0.2) -- (0.6, 1.6);
         \draw  (0.2, 0.2) -- (0.2, 1.2);
        \draw  (1.6, 1.6) -- (1.6, 0.6);
         \draw  (1.2, 1.2) -- (1.2, 0.2);
                \draw  (1.6, 1.6) -- (0.6, 1.6);
         \draw  (1.2, 1.2) -- (0.2, 1.2);

         \node at (-0.1,-0.1) {$x_0$};
                  \node at (1.25,-0.1) {$x_3$};
                  \node at (-0.1,1.25) {$x_1$};
                        \node at (-0.7,3.75) {$x_{-5}$};
                        \node at (0.5,3.75) {$x_{-4}$};
                        \node at (0.95,1.75) {$x_{-3}$};
                         \node at (1.95,1.75) {$x_{-2}$};
                          \node at (0.95,0.85) {$x_{2}$};
                          \node at (3.75,-0.7) {$x_{5}$};
                        \node at (3.95,0.5) {$x_{4}$};
                        \node at (1.95,0.85) {$x_{-1}$};
    \end{tikzpicture} 
    }
    \end{minipage}%
    \begin{minipage}{.5\linewidth}
     \centering
   \subfloat[]{
    \centering
       \begin{tikzpicture}[scale=0.9]
      \begin{scope}[thin, black!40!white]  
      \foreach \i in {0,...,4}
   {
   \draw  (-0.75, \i) -- (4.75, \i);
     \draw  (\i, -0.75) -- (\i, 4.75);
       }
      \end{scope}
      \filldraw (0.2, 0.2) circle (2pt) node[] (){};
       \filldraw (0.2, 1.2) circle (2pt) node[] (){};
       \filldraw (1.2, 0.2) circle (2pt) node[] (){};
       \filldraw (1.2, 1.2) circle (2pt) node[] (){};

        \filldraw (0.5, 1.6) circle (2pt) node[] (){};
         \filldraw (1.6, 0.5) circle (2pt) node[] (){};
          \filldraw (1.6, 1.6) circle (2pt) node[] (){};
           \filldraw (0.5, 3.5) circle (2pt) node[] (){};
              \filldraw (3.5, 0.5) circle (2pt) node[] (){};
                 \filldraw (-0.5, 3.5) circle (2pt) node[] (){};
                    \filldraw (3.5, -0.5) circle (2pt) node[] (){};

     \draw  (3.5, 0.5) -- (3.5, -0.5);
        \draw  (0.5, 3.5) -- (-0.5, 3.5);
        
         \draw  (3.5, 0.5) -- (1.6, 0.5);

                \draw  (0.5, 3.5) -- (0.5, 1.6);
          
         \draw  (0.2, 0.2) -- (1.2, 0.2);
              
         \draw  (0.2, 0.2) -- (0.2, 1.2);
        \draw  (1.6, 1.6) -- (1.6, 0.6);
         \draw  (1.2, 1.2) -- (1.2, 0.2);
                \draw  (1.6, 1.6) -- (0.6, 1.6);
         \draw  (1.2, 1.2) -- (0.2, 1.2);

         \node at (-0.1,-0.1) {$x_0$};
                  \node at (1.25,-0.1) {$x_3$};
                  \node at (-0.1,1.25) {$x_1$};
                        \node at (-0.7,3.75) {$y_{-3}$};
                        \node at (0.5,3.75) {$y_{-2}$};
                        \node at (0.95,1.75) {$y_{-1}$};
                         \node at (1.95,1.75) {$y_{0}$};
                          \node at (0.95,0.85) {$x_{2}$};
                          \node at (3.75,-0.7) {$y_{3}$};
                        \node at (3.95,0.5) {$y_{2}$};
                        \node at (1.95,0.75) {$y_{1}$};
    \end{tikzpicture} 
    }
    \end{minipage}
    \caption{The knot Floer complex of $K_1,$ before and after change of basis.} 
    \label{fig: kjcfk}
\end{figure}

\section{Addendum} \label{sec: addendum}
In this section we provide the code used to obtain the list of four-tuples encoding $(1,1)$ almost L-space knots listed in Table \ref{ta:list}. One can write code that checks the strongly almost coherent condition quite easily, but in practice, we find an algorithm that counts inconsistent arcs in the universal cover more illuminating, and potentially more useful in more general settings. Therefore we provide an algorithm which achieves the latter.   

Given a standard $(1,1)$ diagram associated to $(p,q,r,s)$,   to each intersection point $x$ in the rectangle in Figure \ref{fig: standard} we can assign a three-tuple $[\pos,\dir,\sid]$, such that each intersection point of $\alpha \cap \beta$ in $S^1\times S^1$ has a one-to-two mapping to the three-tuples. 
Here $\pos$ takes integer value,  recording position data in the universal cover, such that $(\pos \operatorname{mod} p)$ specifies the position of $x$ on a given side of the rectangle; $\dir$ takes value in $\pm1$, recording the direction in which the $\beta$ curve intersects $\alpha$ at $x$, where $1$ indicates upwards and $-1$ indicates downwards; $\sid$ takes value in $\pm1$, where $1$ indicates the top side of the rectangle and $-1$ the bottom side.

 When viewing intersection points in the universal cover $\R^2$ (where $S^1\times S^1$ gives a $\Z \times\Z$ lattice), we can assign one more parameter:  $\hei$, which takes integer value,  indicating that (a given lift of) the intersection point $x$ lives in the line $\{y= \hei \}$. 

 The strategy of the algorithm is: we follow a lift $\wb$ of $\beta$ in the universal cover and record the four-tuples $[\pos,\dir,\sid, \hei]$ of intersection points in order. This encodes all the information of the lifted curves. Since $\wb$ is cut by $\wa$ into arcs and two rays, we then simply count the number of arcs of each direction in upper half plane.

 In particular, Algorithm \ref{alg:iseq} gives a function that produces the next four-tuple in the sequence. Using this function, we first obtain a sequence that encodes a portion of the curve $\wb$ in $\R^2$ that corresponds to one iteration of the curve $\beta$ downstairs. In order to include all intersection points on the chosen lift $\wa$, i.e. the line $\{y=0\}$, we have to decide  how many iterations one needs to go through. This is done in the first half of Algorithm \ref{alg:count}.\newline
\begin{algorithm}[H]
\DontPrintSemicolon
\KwData{$(p,q,r,s)$}
\KwResult{A function that with input of the current four-tuple $[\pos,\dir,\sid, \hei]$, outputs the next four-tuple in the sequence. This is used to determine a path of intersection points that $\wb$ travels through in order.}
\SetKwFunction{FMain}{FindNext}
 \SetKwProg{Def}{def}{:}{}
\Def{\FMain{$[\pos,\dir,\sid, \hei]$}}{  \tcp*{Define a function FindNext. Input: current four-tuple;  Output: next four-tuple.} 
\If{$\sid = 1$ and $\dir = 1$}{
\Return $[\pos + s,\dir,-1, \hei]$\;
}
\If {$\sid = -1$ and $\dir = -1$} {
\Return $[\pos - s,\dir,1, \hei]$ \tcp*{Identify two sides of the rectangle}} \;   
\eIf{$\sid = 1$ and $\dir = -1$  \tcp*{Suppose starting from the top side of the rectangle}}     
{\If {$\pos \%p \leq r-1 $ \tcp*{For the first $r$ strands}}     
{\Return{$[\pos + 2*q,\dir,-1, \hei -1 ]$} \;}
\eIf {$r \leq \pos \%p \leq r+ 2*q -1 $ \tcp*{For the rainbow arcs}}     
{\Return{$[2*r + 2*q -1 - 2*\pos \%p +  \pos,1,1, \hei  ]$} \;}
{\Return $[\pos,\dir,-1, \hei -1 ]$  \tcp*{For the remaining strands}} \; 
}
{\tcp*{Suppose starting from the bottom side of the rectangle}
\If {$\pos \%p \leq 2*q-1 $ \tcp*{For the rainbow arcs}}     
{\Return{$[2*q -1 - 2*\pos \%p +  \pos,-1,-1, \hei ]$} \;}
\eIf {$2*q \leq \pos \%p \leq r+ 2*q -1 $ \tcp*{For the next $r$ strands}}     
{\Return{$[\pos - 2*q,\dir,1, \hei +1 ]$} \;}
{\Return $[\pos,\dir,1, \hei +1 ]$  \tcp*{For the remaining strands}} \; 
}
}
\caption{}
\label{alg:iseq}
\end{algorithm}
\begin{algorithm}
\DontPrintSemicolon
\KwData{Continue from Algorithm \ref{alg:iseq}.}
\KwResult{Number of inconsistent arcs in the upper half plane of the universal cover.}
\SetKwFunction{FMain}{FindNext}
\Begin{
Set $\iseq = [[0,1,1,0]]$ \tcp*{Let the starting point be $[0,1,1,0]$} \;
\While{$[\pos,\sid] \neq [0,1]$ \tcp*{Until $\beta$ returns to the original position downstairs} }{ 
$\iseq.\operatorname{\textit{append}}($\FMain{$[\pos,\dir,\sid, \hei]$})\tcp*{Keep adding the next four-tuple to the string}\;
}
Set $\maxhei = \operatorname{max}\{ \iseq.\hei\}$ \tcp*{The maximal height}\;    
Set $\minhei = \operatorname{min}\{ \iseq.\hei\}$ \tcp*{The minimal height}\;
Set $\dh =  \iseq[\length(\iseq)-1].\hei -  \iseq[0].\hei$ \tcp*{The  height difference from the beginning to ending position}\;
\tcc{ Depending on the sign of $\dh$, we will decide number of iterations needed to exhaust all the intersection points on a chosen lift of $\alpha$.}
Set $\fseq = \iseq$ \;
\eIf{$\dh < 0$}
{$\ite = \ceil[big]{\frac{\minhei-\maxhei}{\dh}}$ \;}
 {\eIf{$\dh > 0$}
{$\ite =  \ceil[big]{\frac{\maxhei-\minhei}{\dh}}$ \; }
{$\ite = 0$ \;}
}
\While {$\length (\fseq) \leq \left( \ite  +1 \right) * (\length(\iseq) -1) $ }
{$\fseq.\operatorname{\textit{append}}($\FMain{$[\pos,\dir,\sid, \hei]$}) \;}
\tcc{ For simplicity, next let us only consider when $\dh \geq 0$. The other case is similar.}
For each element of $\fseq,$ set $ \hei = \hei - \maxhei$\; \tcp*{Adjust height such that $\{y=0\}$ is in the middle. We now have the entire portion of $\wb$ that includes all intersection points of $\wb \cap \wa$.} 
Set $[\po,\ne] = [0,0]$ \;
Create a sub-sequence of $\fseq$ with $[\sid,\hei] = [-1,0]$; call it $\inters$.\; 
\tcc{ Recall that $\wa$ cuts $\wb$ into arcs and two rays. We only consider the upper half plane. Note that the condition $\dh \geq 0$ guarantees that the ray in the upper half plane meets $\{y=0\}$ at the last intersection point in the subsequent $\inters.$ }
For {$0\leq i\leq \floor{\frac{\length(\inters)}{2}}$ \newline }  
{\eIf{$\inters[2i].\pos < \inters[2i+1].\pos$  \tcp*{The $(2i+1)$-th and the $(2i+2)$-th points  in the sequence $\inters$ are connected by an arc. }}
{$\po = \po + 1 $\;}
{$\ne = \ne +1 $\;}}
\Return $\operatorname{min}\{\po,\ne \}$ \;
}
\caption{}
\label{alg:count}
\end{algorithm}

\bibliographystyle{amsalpha}
\bibliography{bibliography}
\end{document}